\documentclass[11pt,reqno]{amsart}

\usepackage{graphicx, geometry, stmaryrd, mathrsfs, mathtools, enumitem, bbm, hyperref, epstopdf, amsthm}

\usepackage{framed}

\usepackage{mwe,tikz}
\usepackage[percent]{overpic}
\usepackage{relsize}

\newtheorem{theorem}{Theorem}
\newtheorem{thm}{Theorem}[section]
\newtheorem{prop}[thm]{Proposition}
\newtheorem{lemma}[thm]{Lemma}

\theoremstyle{remark}
\newtheorem{remark}[thm]{Remark}

\theoremstyle{remark}

\theoremstyle{definition}
\newtheorem{definition}[thm]{Definition}

\numberwithin{equation}{section}

\newcommand{\dsp}{\displaystyle}
\newcommand{\GammaRel}{\Gamma_\mathrm{rel}}
\newcommand{\jump}[1]{\left\llbracket{#1}\right\rrbracket}
\makeatletter
\newcommand{\pushright}[1]{\ifmeasuring@#1\else\omit\hfill$\displaystyle#1$\fi\ignorespaces}
\newcommand{\solnSpaceEuler}{\mathscr{S}}
\newcommand{\solnSpace}{\mathscr{S'}}
\makeatother

\makeatother
\newcommand\ackname{Acknowledgements}
\if@titlepage
\newenvironment{acknowledgements}{%
	\titlepage
	\null\vfil
	\@beginparpenalty\@lowpenalty
	\begin{center}%
		\bfseries \ackname
		\@endparpenalty\@M
\end{center}}%
{\par\vfil\null\endtitlepage}
\else

\fi
\makeatother

\def\XXint#1#2#3{{\setbox0=\hbox{$#1{#2#3}{\int}$ }
	\vcenter{\hbox{$#2#3$ }}\kern-.6\wd0}}


\begin{document}
\title[Elliptic Theory and Wind Waves]{Elliptic Equations with Transmission and Wentzell Boundary Conditions and an Application to Steady Water Waves in the Presence of Wind}

\author{Hung Le}
\address{Department of Mathematics, University of Missouri, Columbia, MO 65211}
\email{hdlgw3@mail.missouri.edu}

\begin{abstract}
	In this paper, we present results about the existence and uniqueness of solutions of elliptic equations with transmission and Wentzell boundary conditions. We provide Schauder estimates and existence results in H\"older spaces. As an application, we develop an existence theory for small-amplitude two-dimensional traveling waves in an air-water system with surface tension. The water region is assumed to be irrotational and of finite depth, and we permit a general distribution of vorticity in the atmosphere.
\end{abstract}

\maketitle

\section{Introduction}

\subsection{Elliptic theory}
Let $\Omega \subset \mathbb{R}^n$ be a connected bounded $C^{2, \beta}$ domain for $n > 1$ and $\beta \in (0, 1)$. Suppose that there exists a $C^{2, \beta}$ hypersurface $\Gamma$ that divides $\Omega$ into two connected regions such that 
\begin{align*}
	\Omega = \Omega_1 \cup \Gamma \cup \Omega_2, \qquad 
	\Omega_1 \cap \Omega_2 = \emptyset, \qquad
	\partial\Omega_1 \cap \partial\Omega_2 = \Gamma,
\end{align*}
and denote by $S := \partial\Omega$. Let $\nu = (\nu_1, \dots, \nu_n)$ be the normal vector field on the interface $\Gamma$ pointing outward from $\Omega_1$. We define the co-normal derivative operator on $\Gamma$
\[
	\partial_N := \sum_{i,j=1}^n a^{ij} \nu_i \partial_{x_j},
\]
and the tangential differential operator along $\Gamma$
\[
	\mathcal{D}_s := \sum_{t = 1}^n w^{st} \partial_{x_t}, \quad 1 \le s \le n,
\]
where $w := I_n - \nu \otimes \nu$, and $I_n$ is the $n \times n$ identity matrix.

Our main object of study is the following transmission problem with a Wentzell boundary condition
\begin{equation} \label{Main Eqn} \begin{dcases}
	Lu & = \quad f \quad \text{in } \Omega, \\
	u & = \quad 0 \quad \text{on } S, \\
	\jump{u} & = \quad 0 \quad \text{on } \Gamma, \\
	Bu & = \quad g \quad \text{on } \Gamma,
\end{dcases} \end{equation}
where
\begin{align} \label{L operator}
	Lu & := -\sum_{i,j=1}^n \partial_{x_i}(a^{ij}(x)\partial_{x_j}u) + \sum_{i=1}^n b^i(x)\partial_{x_i}u + c(x)u,
\end{align}
\begin{align} \label{B operator}
	Bu & := - \sum_{s,t=1}^{n} \mathcal{D}_s\big(\mathfrak{a}^{st}(x)\mathcal{D}_t u \big) + \alpha \jump{\partial_N u} + \sum_{s=1}^n \mathfrak{b}^s(x) \mathcal{D}_s u +  \mathfrak{c}(x) u, \quad \alpha = \pm 1.
\end{align}
Here, we are using $\jump{\cdot} := (\cdot)|_{\Omega_1} - (\cdot)|_{\Omega_2}$ to denote the jump operator across $\Gamma$.  We think of $\alpha = +1$ as favorable and $\alpha = -1$ as unfavorable.  We shall assume uniform ellipticity condition on the operators $L$ and $B$; that is, there exist constants $\lambda, \mu > 0$ such that
\begin{align}
	\lambda|\xi|^2 \leq \sum_{i,j=1}^n a^{ij}(x)\xi_i\xi_j \quad \text{for all } x\in\Omega, \xi\in\mathbb{R}^n,
	\label{elliptic cond}
\end{align}
and
\begin{align}
	\mu|\xi|^2 \leq \sum_{s,t=1}^{n} \mathfrak{a}^{st}(x)\xi_s\xi_t  \quad \text{for all } x \in \Gamma \text{ and } \xi \in \mathbb{R}^n \text{ such that } \xi \cdot \nu(x) = 0.
	\label{ellipticity bc}
\end{align}
The coefficients $a^{ij}$ and $\mathfrak{a}^{st}$ satisfy $a^{ij}=a^{ji}, \mathfrak{a}^{st} = \mathfrak{a}^{ts}$ for all $i,j,s,t = 1,\dots,n$. We also assume that $a^{ij}, b^i, c$ are in $L^\infty(\Omega)$, and $\mathfrak{a}^{st}, \mathfrak{b}^s, \mathfrak{c}$ are in $L^\infty(\Gamma)$.

Note that $B$ contains second-order tangential derivatives of $u$. This is characteristic of so-called \emph{Wentzell-type} boundary conditions, whose study was initiated by Wentzell in \cite{Ventcel1959}. They arise, for example, in stochastic equations \cite{Ikeda_Watanabe1989} or as an asymptotic model for roughness of the boundary or other more complex geometrical effects \cite{Bonnaillie_Dambrine2010}. They also appear in water waves and continuum mechanics, which is our principal interest here. For instance, the Young--Laplace Law states that at the interface between two immiscible fluids, the pressure experiences a jump proportional to the curvature. In a free boundary problem where the interface is given as the graph of an unknown function, this naturally leads to quasilinear versions of Wentzell-type conditions. More generally, the curvature of a hyperplane is the first variation of its surface area. Thus, these types of conditions are frequently encountered in free boundary problems where the shape of the interface contributes to the energy.

\emph{Transmission conditions} refers to the jump operator in $B$. They are commonly found in multiphase problems, where physically their purpose is to enforce continuity of the normal stress across a material interface. Many researchers alternatively call these \emph{diffraction problems} (see, for example, \cite{Ladyzhenskaya1968, Arkhipova_Erlhamahmy2002}).

We first have an a priori estimate for classical solutions in H\"older spaces.

\begin{thm}[Schauder estimate] \label{Theorem A priori estimate}
Assume that $a^{ij} \in C^{1,\beta}(\overline{\Omega_1}) \cap C^{1,\beta}(\overline{\Omega_2})$, $b^i, c \in C^{0,\beta}(\overline{\Omega_1}) \cap C^{0,\beta}(\overline{\Omega_2})$, and $\mathfrak{a}^{st} \in C^{1,\beta}(\Gamma)$, $\mathfrak{b}^s, \mathfrak{c} \in C^{0,\beta}(\Gamma)$; and suppose that
\begin{align*}
	\|a^{ij}\|_{C^{1,\beta}(\Omega_k)}, \|b^i\|_{C^{0,\beta}(\Omega_k)}, \|c\|_{C^{0,\beta}(\Omega_k)}, \|\mathfrak{a}^{st}\|_{C^{1,\beta}(\Gamma)}, \|\mathfrak{b}^s\|_{C^{0, \beta}(\Gamma)}, \|\mathfrak{c}\|_{C^{0,\beta}(\Gamma)} &< \Lambda_2
\end{align*}
for some constant $\Lambda_2 > 0$, for all $i,j,s,t = 1, \dots, n$, and $k = 1, 2$. Suppose that $u \in C^0(\overline{\Omega}) \cap C^{2,\beta}(\overline{\Omega_1}) \cap C^{2,\beta}(\overline\Omega_2)$ solves equation \eqref{Main Eqn} with $\alpha = \pm 1$.  Then for any $\Omega' \subset\subset \overline\Omega \backslash (\Gamma \cap S)$, if $f \in C^{0,\beta}(\overline{\Omega_1}) \cap C^{0,\beta}(\overline{\Omega_2})$ and $g \in C^{0,\beta}(\Gamma)$, the following estimate holds:
\begin{align} \label{Estimate Theorem A priori Holder}
	\|u_1\|_{C^{2,\beta}(\Omega'_1)} + &\|u_2\|_{C^{2,\beta}(\Omega'_2)} \\
	&\le C \left(\|u\|_{C^0(\Omega)} + \|f\|_{C^{0,\beta}(\Omega_1)} + \|f\|_{C^{0,\beta}(\Omega_2)} + \|g\|_{C^{0,\beta}(\Gamma)} \right) \nonumber
\end{align}
for some constant $C=C(n,\beta,\Lambda_2,\lambda,\mu,\Omega'_1,\Omega'_2) > 0$, where $\Omega'_k := \Omega' \cap \Omega_k$ and $u_k := u|_{\Omega_k}$.
\end{thm}

Note that the estimates in the theorem hold on subsets that are positively separated from $\Gamma \cap S$ where the boundary may not be smooth. The next theorems will assume that $\Omega \subset \mathbb{T}^{n-1} \times \mathbb{R}$, where $\mathbb{T}$ is a torus and $S \cap \Gamma = \emptyset$.

\begin{thm}[Existence and uniqueness of solutions in $C^{2,\beta}$]
\label{Theorem Existence and uniqueness of C^(2,beta) with c >= 0}
Suppose the coefficients $a^{ij}, b^i, c, \mathfrak{a}^{st}, \mathfrak{b}^s, \mathfrak{c}$ exhibit the same regularity as in Theorem \ref{Theorem A priori estimate} and assume in addition that $c \ge 0$ and $\mathfrak{c} > 0$. Then for all $f \in C^{0,\beta}(\overline{\Omega_1}) \cap C^{0,\beta}(\overline{\Omega_2})$ and $g \in C^{0,\beta}(\Gamma)$, the problem \eqref{Main Eqn} with $\alpha=1$ has a unique $C^{0,\beta}(\overline{\Omega}) \cap C^{2,\beta}(\overline{\Omega_1}) \cap C^{2,\beta}(\overline{\Omega_2})$ solution.  
\end{thm}

We emphasize that this theorem only holds if we assume $c \ge 0$, $\mathfrak{c} > 0$, and $\alpha$ has a favorable sign.  However, for a general $c$, $\mathfrak{c}$, and $\alpha$, we are still able to assert the Fredholm solvability of the problem. Letting
\begin{equation} \label{defn X}
	X := C^{2,\beta}(\overline{\Omega_1}) \cap C^{2,\beta}(\overline{\Omega_2}) \cap \{u|_S=0\} \cap C^{0,\beta}(\overline{\Omega}),
\end{equation}
which is a Banach space with respect to the norm
\[
	\|u\|_X := \|u\|_{C^{2,\beta}(\Omega_1)} + \|u\|_{C^{2,\beta}(\Omega_2)} + \|u\|_{C^{0,\beta}(\Omega)},
\]
we have the following theorem.

\begin{thm}[Fredholm solvability] \label{Theorem Fredholm solvability}
Suppose the coefficients $a^{ij}, b^i, c, \mathfrak{a}^{st}, \mathfrak{b}^s, \mathfrak{c}$ exhibit the same regularity as in Theorem \ref{Theorem A priori estimate} and $\alpha = \pm 1$. Then either
\begin{enumerate}[label=\alph*)]
	\item the homogeneous problem \eqref{Main Eqn} with $f = g = 0$ has nontrivial solutions that form a finite dimensional subspace of $X$, or 
	\item the homogeneous problem has only the trivial solution in which case the inhomogeneous problem has a unique solution in $X$ for all $f \in C^{0,\beta}(\overline{\Omega_1}) \cap C^{0,\beta}(\overline{\Omega_2})$ and $g \in C^{0,\beta}(\Gamma)$.
\end{enumerate}
\end{thm}

A great deal of research has been devoted to studying elliptic problems with linear and nonlinear Wentzell boundary conditions, but they remain comparatively less well-understood. One of the earliest works to consider Wentzell conditions was Korman \cite{Korman1986} who, like us, was interested in their connection to a problem in water waves. Specifically, he investigated a model describing three-dimensional periodic capillary-gravity waves where the gravity pointed upward. A Schauder theory was later provided by Luo and Trudinger \cite{Luo_Trudinger1991} for the linear case. In the quasilinear setting, Luo \cite{Luo1991} gave a priori estimates for uniformly elliptic Wentzell conditions, while later Luo and Trudinger \cite{Luo_Trudinger1994} studied the degenerate case. More recently, Nazarov and Paletskikh \cite{Nazarov_Paletskikh2015} derived local H\"older estimates in the spirit of De Giorgi for divergence form elliptic equations with measurable coefficients and a Wentzell condition imposed on a portion of the boundary. See also the survey by Apushkinskaya and Nazarov \cite{Apushkinskaya_Nazarov2000} for a summary of the progress made on the nonlinear problem.  

Transmission boundary conditions are of great importance to physics and other applied sciences. They are also of interest from a purely mathematical perspective as they arise naturally in the weak formulation of PDEs with discontinuous coefficients. The study of transmission problems dates back to the 1950s and 1960s. Schechter \cite{Schechter1960} and \u{S}eftel' \cite{Seftel1963} investigated even-order elliptic equations on a smooth and bounded domain with smooth coefficients. Schechter obtained estimates and provided an existence for weak solutions. His strategy involved transforming the transmission problem into a mixed boundary value problem for a system of equations. On the other hand, \u{S}eftel' found a priori $L^p$-estimates. Ole\u{\i}nik \cite{Oleinik1961} also studied transmission problems for second-order elliptic equations with smooth coefficients; approximating equations were used to derive results for weak solutions. One of the most foundational work was done by Ladyzhenskaya and Ural'tseva \cite{Ladyzhenskaya1968}, who considered second-order elliptic equations on a bounded domain and then obtained estimates for weak and classical solutions in Sobolev and H\"older spaces, respectively. In contrast to Schechter's approach, Ladyzehnskaya and Ural'tseva exploited cleverly chosen test functions to deduce their a priori estimates. More recently, Borsuk \cite{Borsuk2008,Borsuk2009,Borsuk2010} has treated linear and quasilinear transmission problems on non-smooth domains.

Apushkinskaya and Nazarov \cite{Apushkinskaya_Nazarov2001} considered Sobolev and H\"older solutions of linear elliptic and parabolic equations for two-phase systems.  However, they only examined the problem with a favorable sign $\alpha = +1$ of the transmission term, and did not study Fredholm property.  Note that in water wave applications, the sign is typically unfavorable.  With that in mind, in this paper we make the effort to also include Schauder estimates and Fredholm solvability for $\alpha = -1$ as well; see also Remarks \ref{Remark sign jump co-normal term}, \ref{Remark Fredholm index 0}, and \ref{Remark sign of B switched}.  Our approach is to view the Wentzell boundary condition as a non-local $(n-1)$-dimensional elliptic equation, treating the jump in the co-normal derivative term as forcing that can be controlled using techniques from the literature on transmission problems.

\subsection{Steady wind-driven capillary-gravity water waves}
Our second set of results considers an application of the above elliptic theory to a problem in water waves. In particular, we will prove the existence of small amplitude periodic wind-driven capillary-gravity waves in a two-phase air-water system. One of the main novelties here is that we also allow for a general distribution of vorticity in the air region. For simplicity we take the flow in the water to be irrotational. When discussing these results, we adopt notational conventions common in studies of steady water waves which occasionally conflict with our notations in the elliptic theory part.

\begin{figure}[!ht]
	\centering
	\begin{overpic}[scale=.2]{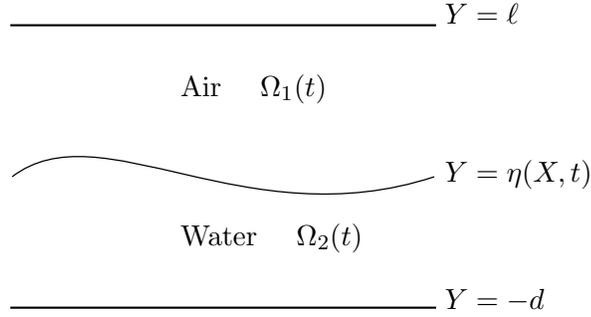}
		\put (40, 50) {Air \quad $\Omega_1(t)$}
		\put (40, 15) {Water \quad $\Omega_2(t)$}
		\put (102, 0) {$Y = -d$}
		\put (102, 30) {$Y = \eta(X, t)$}
		\put (102, 67) {$Y = \ell$}
	\end{overpic}
	\caption{The air-water system}
	\label{fig:air_water_domain}
\end{figure}

Let us now formulate the problem more precisely. Fix a Cartesian coordinate system $(X,Y) \in \mathbb{R}^2$ so that the $X$-axis points in the direction of wave propagation and the $Y$-axis is vertical. The ocean bed is assumed to be flat and at the depth $Y = -d$, while the interface between the water and the atmosphere is a free surface given as the graph of $\eta = \eta(X, t)$. We then normalize $\eta$ so that the free surface is oscillating around the line $Y = 0$. The atmospheric domain is assumed to be bounded in $Y$; that is, the air region lies below $Y = \ell$ for some fixed $\ell > 0$. At a given time $t$, the fluid domain is
\[
	\Omega(t) = \Omega_1(t) \cup \Omega_2(t),
\]
where $\Omega_1$ is the air region,
\[
	\Omega_1(t) := \{ (X, Y) \in \mathbb{R}^2: \eta(X, t) < Y < \ell \},
\]
and $\Omega_2$ is the water region,
\[
	\Omega_2(t) := \{ (X, Y) \in \mathbb{R}^2: -d < Y < \eta(X, t) \}.
\]
We also denote $\mathcal{I}(t) := \partial\Omega_1(t) \cap \partial\Omega_2(t)$. Here we think of $\mathcal{I}(t)$ as playing the role of $\Gamma$ in the notation of the previous subsection.

Let $u = u(X, Y, t)$ and $v = v(X, Y, t)$ be the horizontal and vertical fluid velocities, respectively, and denote by $P = P(X, Y, t)$ the pressure.  We say that this is a \emph{traveling wave} provided that there exists a wave speed $c > 0$ such that the change of variables
\[
	(X,Y) \mapsto (x,y) := (X-ct, Y)
\]
eliminates time dependence.  The velocity field is assumed to be incompressible and, in the moving frame, $(u, v, \eta, P)$ are taken to be $2\pi$-periodic in $x$.

For water waves, the governing equations are the incompressible steady Euler system: 
\begin{equation} \label{Eulerian eqn 1}
	\begin{dcases}
		u_x + v_y &= 0, \\
		\varrho(u-c)u_x + \varrho v u_y &= -P_x, \\
		\varrho(u-c)v_x + \varrho v v_y &= -P_y -g\varrho,
	\end{dcases}
	\qquad \text{in } \Omega
\end{equation}
where $g > 0$ is the gravitational constant, and $\varrho = \varrho_1 \mathbbm{1}_{\Omega_1} + \varrho_2 \mathbbm{1}_{\Omega_2}$ with $\varrho_1$ and $\varrho_2$ assumed to be constant densities of $\Omega_1$ and $\Omega_2$, respectively. We assume $\varrho_1 < \varrho_2$. The symbol $\mathbbm{1}_{\Omega_i}$ stands for the characteristic function on $\Omega_i$. As above, $\jump{\cdot}$ denotes the jump over $\mathcal{I}$, that is, $\jump{\cdot} = (\cdot)|_{\Omega_1} - (\cdot)|_{\Omega_2}$.

The kinematic and dynamic boundary conditions for the lidded atmosphere problem with surface tension $\sigma$ are
\begin{equation} \label{Eulerian eqn 2}
	\begin{dcases}
		v = 0 & \text{on } y = \ell, \\
		v = 0 & \text{on } y = -d, \\
		v = (u - c)\eta_x &\text{on } y = \eta(x), \\
		\jump{P} = -\sigma \frac{\eta_{xx}}{\left( 1 + (\eta_x)^2 \right)^{3/2}} &\text{on } y = \eta(x).
	\end{dcases}
\end{equation}
Note that the last condition will give rise to nonlinear Wentzell and transmission terms. In particular, the right hand side can be viewed as a second-order elliptic operator acting on $\eta$, while the jump in the pressure will relate to a jump in $(u-c)^2+v^2$ via Bernoulli's theorem that we discuss below.

We consider waves without (horizontal) stagnation, that is, we will always assume
\begin{equation} \label{wo stagnation}
	u - c < 0 \quad \text{in } \overline{\Omega}.
\end{equation}
As $(u,v)$ is divergence free according to \eqref{Eulerian eqn 1}, we can define the \emph{pseudostream} function $\psi = \psi(x,y)$ for the flow by
\begin{equation} \label{pseudostream function defn}
	\psi_x = \sqrt{\varrho} v, \qquad \psi_y = \sqrt{\varrho} (u - c) \qquad \text{in } \Omega.
\end{equation}
The level sets of $\psi$ are called \emph{streamlines}. Without stagnation \eqref{wo stagnation}, we have $\psi_y < 0$, which implies that each streamline is given as the graph of a function of $x$ via a simple Implicit Function Theorem argument. The boundary conditions in \eqref{Eulerian eqn 2} show that the air-water interface, bed, and lid are each level sets of $\psi$. We will take $\psi = 0$ on the upper lid so that $\psi = -p_0$ on $y = -d$, where $p_0$ is defined by
\[
	p_0 := \int_{-d}^{\eta(x)} \sqrt{\varrho(x,y)} (u(x, y) - c) \,\mathrm{d}y.
\]
It can be shown that $p_0$ does not depend on $x$ (see, for example, \cite{Walsh2009}). Bernoulli's theorem states that
\[
	E := P + \frac{\varrho}{2}((u - c)^2 + v^2) + g \varrho y
\]
is constant along streamlines. Evaluating the jump of $E$ on the interface gives
\[
	\jump{|\nabla \psi|^2} + 2g\jump{\varrho}(\eta + d) + \sigma \kappa = Q \quad \text{on } y = \eta(x),
\]
where $\kappa$ is the signed curvature of the air-water interface and $Q := 2\jump{E + g \varrho d}$.

Recall that in two dimensions, the \emph{vorticity} $\omega$ is defined to be
\[
	\omega := v_x - u_y.
\]
If there is no stagnation \eqref{wo stagnation}, there exists a function $\gamma$, called the \emph{vorticity strength function}, such that
\[
	\omega(x,y) = \gamma(\psi(x, y)) \quad \text{for all } (x,y) \in \Omega.
\]
The vorticity plays a key role in the wind generation of water waves as we will discuss below.  Mathematically, it substantially complicates the analysis.

Finally, we will use the following notational conventions. For any integer $k \ge 0$, $\alpha \in (0, 1)$, and an open region $R \subset \mathbb{R}^n$, we define the space $C^{k+\alpha}_\text{per}(\overline{R})$ to be the set of $C^{k+\alpha}(\overline{R})$ functions that are $2\pi$-periodic in their first argument. 

Our main theorem is an existence result for traveling capillary-gravity water waves in the presence of wind.

\begin{thm}[Existence of small amplitude wind-driven water waves] \label{Theorem existence of small amplitude wind-drive water waves}
Fix $d$, $\ell$, $c$ $> 0$, and $p_0 < p_1 < 0$. For any vorticity function $\gamma \in C^{0, \alpha}([p_1, 0])$ and $\sigma > 0$ sufficiently large, there exists a $C^1$ curve
\[
	\mathcal{C}_{\mathrm{loc}}' := \{ (u(s), v(s), \eta(s), Q(s)): s \in (-\epsilon, \epsilon) \}
\]
of traveling wave solutions to the capillary-gravity water wave problem \eqref{Eulerian eqn 1}--\eqref{wo stagnation} such that
\begin{enumerate} [label=\alph*)]
	\item Each $(u, v, \eta, Q) \in \mathcal{C}_\mathrm{loc}'$ is of class
	\[
		\left( u(s), v(s), \eta(s), Q(s) \right) \in \left( C_{\mathrm{per}}^\alpha(\overline{\Omega}) \cap C_{\mathrm{per}}^{1 + \alpha}(\overline{\Omega}(s) \backslash \mathcal{I}(s)) \right)^2 \times C_{\mathrm{per}}^{2 + \alpha}(\mathbb{R}) \times \mathbb{R} =: \solnSpaceEuler,
	\]
	where $u(s)$ and $v(s)$ are even and odd in the first coordinate, respectively, $\eta(s)$ is even in $x$, and $\Omega(s)$ is the domain corresponding to $\eta(s)$;
	
	\item $\left( u(0), v(0), \eta(0), Q(0) \right) = (U_*(y), 0, 0, Q_*)$, where $(U_*, Q_*)$ is laminar solution.
\end{enumerate}
\end{thm}

We prove this theorem using a local bifurcation theoretic strategy that draws on the ideas of Constantin and Strauss \cite{Constantin_Strauss2004}, who studied rotational periodic gravity water waves in a single fluid. Indeed, following the publication of \cite{Constantin_Strauss2004},  traveling water waves with vorticity have been an extremely active area of research (see, for example, the surveys in \cite{Strauss2010,Constantin2011}).

Our most direct influence is the work of B\"uhler, Shatah, and Walsh \cite{Buhler_Shatah_Walsh2013} on the existence of steady gravity waves in the presence of wind. These authors studied exactly the system \eqref{Eulerian eqn 1}--\eqref{wo stagnation} taking $\sigma = 0$. One of the main objectives of that paper was to construct waves that were dynamically accessible from an initial state where the flow is laminar and the horizontal velocity experiences a jump over the interface. More specifically, this meant that the circulation along each streamline was prescribed in order to ensure that its values in the air and water regions were distinct (see Remark \ref{Remark Gamma_Rel}). We also adopt this approach in the present paper, though the addition of surface tension necessitate many nontrivial adaptations.

\subsection{History of the problem} \label{History of the problem section}

Steady capillary and capillary-gravity waves have been the subject of extensive research. Because we are particularly interested in the role of vorticity, we will restrict our discussion to rotational waves. In this setting, progress is much more recent and begins with the work of Wahl\'en \cite{Wahlen_gravity2006,Wahlen2006}, who proved the existence of small-amplitude periodic capillary and capillary-gravity waves in two-dimensions for a single fluid system. As in \cite{Constantin_Strauss2004}, this was done for a general vorticity function $\gamma$. Contrary to the gravity wave case, Wahl\'en showed that with surface tension there can be double bifurcation points; this is a rotational analogue of the famous Wilton ripples \cite{Wilton1915}. Later, Walsh considered two-dimensional periodic capillary-gravity waves with density stratification \cite{Walsh2014,Walsh2014_2}.  

Recently, Martin and B-V Matioc proved the existence of steady small-amplitude capillary-gravity water waves with piecewise constant vorticity \cite{Martin_Matioc2014}.  While they consider a one-layer model, the analysis has a similar flavor to that in the present paper.  A-V Matioc and B-V Matioc also constructed weak solutions for steady capillary-gravity water waves in a single fluid \cite{Matioc_Matioc2014}.  

The waves we construct can also be viewed as internal waves moving along the interface between two immiscible fluid layers confined in a channel. Versions of this problem have been investigated by many authors. For instance, Amick--Turner \cite{Amick_Turner1986} and Sun \cite{Sun1997,Sun2002} considered the existence of solitary waves in a channel where the flow is irrotational at infinity.  Amick--Turner built their solitary waves as limits of periodic waves with the period tending to infinity. Sun, on the other hand, exploited the fact that the leading-order form of the wave is given by the Benjamin--Ono equation, and then used singular integral operator estimates to control the remainder. The existence of continuously stratified channel flows has also been verified in a number of regimes. Note that these are rotational, since heterogeneity in the density produces vorticity. Specifically, Turner \cite{Turner1981} and Kirchg\"assner \cite{Kirchgassner1982} investigated small-amplitude continuously stratified waves using a variational scheme and a center manifold reduction method, respectively. A large-amplitude existence theory  was also provided by Bona, Bose, and Turner \cite{Bona_Bose_Turner1983}, Lankers and Friesecke \cite{Lankers_Friesecke1997}, and Amick \cite{Amick1984}. We remark that, in all of these works, the vorticity vanishes at infinity. Finally, internal waves with surface tension on the interface were recently considered by Nilsson \cite{Nilsson2017}. In that paper, each fluid layer was assumed to be irrotational and constant density. Using spatial dynamics and a center manifold reduction, Nilsson proved the existence of both periodic and solitary wave solutions.

As mentioned above, steady water waves in the presence of wind was studied by B\"uhler, Shatah, and Walsh in \cite{Buhler_Shatah_Walsh2013}. Our main contribution relative to that work is to account for capillary effects on the air-water interface. It is known that surface tension is important in the formation of wind-driven waves. Indeed, high frequency and small-amplitude capillary-gravity waves are the first to form when wind blows over a quiescent body of water.

One of the most successful explanations for the mechanism behind the wind generation of water waves was given by Miles \cite{Miles1957}. His main observation was that vorticity in the air region can create a certain resonance phenomenon that destabilizes the system. Importantly, this so-called \emph{critical layer instability} can occur even when the horizontal velocity is continuous --- or nearly continuous --- over the interface, and therefore does not require exceedingly strong wind speeds like the Kelvin--Helmholtz model. The mathematical ideas underlying Miles's theory were recently reexamined and rigorously proved by B\"uhler, Shatah, Walsh, and Zeng \cite{Buhler_Shatah_Walsh_Zeng2016}. In that work, the authors also allowed surface tension. This is somewhat important as the interface Euler problem itself is ill-posed when there is a jump in the tangential velocity and there is no surface tension (see, for example, \cite{Beale_Hou_Lowengrub1993}).  In a forthcoming work, the author intends to study the stability of the family of waves constructed in Theorem \ref{Theorem existence of small amplitude wind-drive water waves}.  This will serve as a model for wind generation of water waves in the spirit of Miles, but with an initial state that is not purely laminar.

\subsection{Plan of the article}

We now briefly discuss the strategies we use to derive these results. The elliptic theory is proved in Section \ref{elliptic section}. Our approach is based on the work of Luo and Trudinger \cite{Luo_Trudinger1991}, who gave Schauder estimates for elliptic equations with Wentzell boundary conditions.

In Section \ref{existence section}, we construct capillary-gravity water waves where the air region is rotational. Following B\"uhler, Shatah, and Walsh \cite{Buhler_Shatah_Walsh2013}, the first step in this procedure is to reformulate the interface Euler system \eqref{Eulerian eqn 1}--\eqref{wo stagnation} as a quasilinear elliptic equation on a fixed domain. Due to surface tension, there is now a nonlinear Wentzell condition on the image of the interface in these new coordinates. We construct the non-laminar waves using local bifurcation theory.  This entails studying the spectrum of the linearized equation at a laminar flow, and here we make essential use of the elliptic theory developed in Section \ref{elliptic section}. One major difficulty that arises is that this linearized problem is of Sturm--Liouville type, but associated to an indefinite inner product. Consequently, to successfully determine the spectral behavior, we must work in Pontryagin spaces. A similar issue was encountered by Wahl\'en in \cite{Wahlen_gravity2006,Wahlen2006}. Finally, we apply the Crandall--Rabinowitz local bifurcation theorem to obtain Theorem \ref{Theorem existence of small amplitude wind-drive water waves}.

\section{Elliptic Theory} \label{elliptic section}
To simplify subsequent calculations, it is convenient to first change variables. Fix a point $x^0 \in \Gamma$. Then by the assumption on $\Omega$, there is a neighborhood $\mathcal{U}$ of $x^0$ and a $C^{2,\beta}$ diffeomorphism that maps $\mathcal{U}$ to some ball $B \subset \mathbb{R}^n$ so that $\Gamma$ maps to $\{x_n = 0\}$, $\Omega_1$ to $B \cap \{x_n > 0\}$, and $\Omega_2$ to $B \cap \{x_n < 0\}$ (see, for example, \cite{Luo_Trudinger1991}). Then it suffices to assume that $\Gamma$ is the hyper-plane $\{x_n = 0\}$, and consequently, $\Omega_1$ and $\Omega_2$ lie inside the upper-half and lower-half planes respectively.

In this case, the co-normal derivative operator simplifies to
\[
	\partial_N u = -\sum_{j=1}^n a^{nj}\partial_{x_j}u,
\]
and the Wentzell and transmission condition on $\Gamma$ becomes
\[
	Bu = -\sum_{s,t=1}^{n-1} \partial_{x_s}\left(\mathfrak{a}^{st}\partial_{x_t}u\right) + \alpha \jump{\partial_N u} + \sum_{s=1}^{n-1} \mathfrak{b}^s \partial_{x_s}u + \mathfrak{c} u.
\]
We also denote by $\nabla'$ the tangential gradient on $\Gamma$ in this case.

\subsection{Classical solutions}
First, we prove our theorem on Schauder estimates for solutions in H\"older spaces. This relies on the observation that one can apply $(n-1)$-dimensional elliptic estimates for $B$ on $\Gamma$ with transmission boundary condition being lower ordered.

\begin{proof}[Proof of Theorem \ref{Theorem A priori estimate}]
Using the above change of variables, we rewrite the condition on $\Gamma$
\[
	-\sum_{s,t=1}^{n-1} \mathfrak{a}^{st} \partial_{x_s}\partial_{x_t}u - \sum_{s,t=1}^{n-1} (\partial_{x_s} \mathfrak{a}^{st})(\partial_{x_t}u) - \alpha \sum_{j=1}^n a^{nj}\partial_{x_j}u_1 + \alpha \sum_{j=1}^n a^{nj}\partial_{x_j}u_2 + \sum_{s=1}^{n-1} \mathfrak{b}^s \partial_{x_s}u + \mathfrak{c} u = g.
\]
We then cover $\Gamma$ by a finite number of spheres in which the estimate in \cite[Theorem 6.2]{Gilbarg_Trudinger2001} for $B$ on $\Gamma$ can be applied. This ensures the existence of a positive constant $C = C(n, \beta, L, B, \mu)$ such that
\begin{align} \label{u1 est on Gamma}
	\|u_1\|_{C^{2,\beta}(\Gamma')} \le C(\|u_1\|_{C^0(\Gamma)} + \|u_2\|_{C^{1,\beta}(\Gamma')} + \|g\|_{C^{0,\beta}(\Gamma)}).
\end{align}
Similarly, we have
\begin{align} \label{u2 est on Gamma}
	\|u_2\|_{C^{2,\beta}(\Gamma')} \le C(\|u_2\|_{C^0(\Gamma)} + \|u_1\|_{C^{1,\beta}(\Gamma')} + \|g\|_{C^{0,\beta}(\Gamma)}).
\end{align}
Next, we use a basic elliptic estimate for the Dirichlet problem in $\Omega'_k$ with boundary condition $u_k|_\Gamma$ (see, for example, \cite[Theorem 6.6]{Gilbarg_Trudinger2001}), to obtain
\begin{align} \label{u_i est in Omega}
	\|u_k\|_{C^{2,\beta}(\Omega'_k)} \le C \left( \|u_k\|_{C^0(\Omega_k)} + \|u_k\|_{C^{2,\beta}(\Gamma')} + \|f\|_{C^{0,\beta}(\Omega_k)} \right).
\end{align}
Moreover, we have the following interpolation
\begin{align} \label{interpolation u C(1,beta)}
	\|u_k\|_{C^{1,\beta}(\Gamma')} \le C_\epsilon \|u\|_{C^0(\Gamma)} + \epsilon \|u_k\|_{C^{2,\beta}(\Gamma')}
\end{align}
for some $\epsilon > 0$. Finally, evaluating \eqref{u_i est in Omega} with $k=1,2$ and summing, using $\jump{u}|_\Gamma=0$ and the estimates \eqref{u1 est on Gamma}, \eqref{u2 est on Gamma}, \eqref{interpolation u C(1,beta)} and choosing appropriate $\epsilon$ give
\[
	\|u_1\|_{C^{2,\beta}(\Omega'_1)} + \|u_2\|_{C^{2,\beta}(\Omega'_2)} \le C \left( \|u\|_{C^0(\Omega)} + \|f\|_{C^{0,\beta}(\Omega_1)} + \|f\|_{C^{0,\beta}(\Omega_2)} + \|g\|_{C^{0,\beta}(\Gamma)} \right). \qedhere
\]
\end{proof}

\begin{remark} \label{Remark sign jump co-normal term}
A version of Theorem \ref{Theorem A priori estimate} was stated in \cite[Theorem $2.3^*$]{Apushkinskaya_Nazarov2001} without proof under the assumption that $\alpha = +1$ in the boundary operator \eqref{B operator}.  However, according to the above proof, this theorem holds regardless of the sign of $\jump{\partial_N u}$.  
\end{remark}

Next, in preparation for proving the existence and uniqueness result, we first establish a maximum principle.  Apushkinskaya and Nazarov state a similar result in \cite[Theorem 3.1]{Apushkinskaya_Nazarov2001}.  Using our notations, we have the lemma.  

\begin{lemma}[Maximum Principle] \label{Lemma Max Principle elliptic theory}
Suppose the coefficients $a^{ij}, b^i, c, \mathfrak{a}^{st}, \mathfrak{b}^s, \mathfrak{c}$ exhibit the same regularity as in Theorem \ref{Theorem A priori estimate} and assume in addition that $c \ge 0$ and $\mathfrak{c} > 0$.  Let $\alpha = 1$ and suppose that $u \in C^0(\overline{\Omega}) \cap C^2(\overline{\Omega_1}) \cap C^2(\overline{\Omega_2})$ satisfies
\[
	Lu \le f \quad \text{in } \Omega, \qquad u = 0 \quad \text{on } S, \qquad Bu \le g \quad \text{on } \Gamma.
\]
Then we have the estimate
\begin{align} \label{a priori est wrt f and g}
	\sup_\Omega u \le \sup_\Gamma \left| \frac{g}{\mathfrak{c}} \right| + C \sup_\Omega \left| \frac{f}{\lambda} \right|
\end{align}
for some positive constant $C=C(\mathrm{diam}\ \Omega, \lambda, \|\partial_{x_i}a^{ij}\|_{L^\infty}, \|b^i\|_{L^\infty})$, where the $L^\infty$ norms are taken over $\Omega_1$ and $\Omega_2$.  
\end{lemma}

\begin{proof}
We will follow very closely the classical arguments when proving this maximum principle in the interior. Rewriting $L$ in non-divergence form gives
\[
	Lu = -\sum_{i,j=1}^n a^{ij}\partial_{x_i}\partial_{x_j}u + \sum_{i=1}^n \tilde{b}^i \partial_{x_i}u + cu,
\]
where $\tilde{b} := b^i - \partial_{x_i}a^{ij}$. Setting 
\[
	\tau := \frac{\|\tilde{b}^i\|_{L^\infty(\Omega_1)} + \|\tilde{b}^i\|_{L^\infty(\Omega_2)}}{\lambda},
\]
choosing $\sigma \ge 1$ large enough so that $\sigma^2 - \tau\sigma \ge 1$, and without loss of generality, because of the boundedness of $\Omega$, assuming $\Omega$ lies between $\{x_1 = 0\}$ and $\{x_1 = d\}$, let
\[
	v := \sup_{\partial\Omega_k} u^+ + \left( e^{\sigma d} - e^{\sigma x_1} \right) \sup_{\Omega_k} \frac{f^+}{\lambda},
\]
where $u^+ := \max(u,0)$ and $d := \text{diam } \Omega$. Then
\[
	Le^{\sigma x_1} = (-a^{11}\sigma^2 + \tilde{b}^1\sigma + c)e^{\sigma x_1} \le -\lambda(\sigma^2 - \tau\sigma)e^{\sigma x_1} + c\,e^{\sigma x_1} \le -\lambda + c\,e^{\sigma x_1}.
\]
Then since $c \ge 0$,
\begin{align*}
	Lv &\ge c\,\sup_{\partial\Omega_k}u^+ + c\,e^{\sigma d}\sup_{\Omega_k}\frac{f^+}{\lambda} - (-\lambda + c\,e^{\sigma x_1}) \sup_{\Omega_k} \frac{f^+}{\lambda} \\
	&\ge c(e^{\sigma d} - e^{\sigma x_1}) \sup_{\Omega_k} \frac{f^+}{\lambda} + \sup_{\Omega_k} f^+ \ge \sup_{\Omega_k} f^+,
\end{align*}
so we have $L(u-v) \le 0$ in $\Omega_k$. On the other hand, by construction $u - v \le 0$ on $\partial\Omega_k$. Therefore, the maximum principle implies $u \le v$ in $\Omega_k$, and that there exists a positive constant $C = C(d, \lambda, \|\partial_{x_i}a^{ij}\|_{L^\infty}, \|b^i\|_{L^\infty})$ such that
\[
	\sup_{\Omega_k} u \le  \sup_{\partial\Omega_k} u^+ + C \sup_{\Omega_k} \frac{f^+}{\lambda} \quad \text{for } k = 1,2.
\]
Next, since $u|_S = 0$, if $u|_\Gamma \le 0$ for all $x \in \Gamma$, then
\[
	\sup_{\partial\Omega_k} u^+ = 0 \le \sup_\Gamma \left| \frac{g}{\mathfrak{c}} \right|.
\]
If we suppose that $u$ attains its local maximum at some point $x_0 \in \Gamma$ and $u(x_0) > 0$, then by the positive definiteness of the matrix $(\mathfrak{a}^{st})$,
\[
	\nabla' u(x_0) = 0 \quad \text{and} \quad \sum_{s,t=1}^{n - 1} (\mathfrak{a}^{st}\partial_{x_s}\partial_{x_t}u)(x_0) \le 0.
\]
By the positive-definiteness of the matrix $(a^{ij})$, we have
\begin{align*}
	\partial_N u_1(x_0) &= -(a^{nn}\partial_{x_n}u_1)(x_0) \ge 0 \\
	\partial_N u_2(x_0) &= -(a^{nn}\partial_{x_n}u_2)(x_0) \le 0,
\end{align*}
and hence $\jump{\partial_N u(x_0)} \ge 0$.  Then the condition on $\Gamma$ gives
\[
	(\mathfrak{c} u)(x_0) \le g(x_0) + \sum_{s,t=1}^{n-1}(\mathfrak{a}^{st}\partial_{x_s}\partial_{x_t}u)(x_0) - \jump{\partial_N u(x_0)} \le g(x_0),
\]
so since $\mathfrak{c} > 0$ for all $x \in \Gamma$, we obtain
\[
	\sup_\Gamma u = u(x_0) \le \frac{g(x_0)}{\mathfrak{c}(x_0)} \le \sup_\Gamma \frac{g}{\mathfrak{c}}.
\]
Therefore,
\[
	\sup_{\partial\Omega_k} u^+ = \sup_\Gamma u \le \sup_\Gamma \left| \frac{g}{\mathfrak{c}} \right|,
\]
and hence we obtain the desired estimate \eqref{a priori est wrt f and g} by using $\dsp \sup_\Omega u = \max(\sup_{\Omega_1} u, \sup_{\Omega_2} u)$.
\end{proof}

\begin{remark} \label{Remark Omega periodic Hold holds}
Note that if $\Omega$ is periodic in one variable, the lemma still holds by modifying the proof to assume that $\Omega$ lies between two hyperplanes parallel to the periodic direction.  
\end{remark}

Using the notation of a H\"older seminorm, we have the following simple lemma whose proof will be omitted:

\begin{lemma} \label{Lemma showing u in C^0,beta}
Suppose $u \in C^0(\overline{\Omega}) \cap C^{0,\beta}(\overline{\Omega_1}) \cap C^{0,\beta}(\overline{\Omega_2})$. Then $[u]_{0,\beta;\overline{\Omega}}$ is finite, and
\[
	\|u\|_{C^{0,\beta}(\Omega)} \le C\big( \|u_1\|_{C^{0,\beta}(\Omega_1)} + \|u_2\|_{C^{0,\beta}(\Omega_2)} \big).
\]
\end{lemma}

Now we can derive the existence and uniqueness of solution in H\"older spaces.
\begin{proof}[Proof of Theorem \ref{Theorem Existence and uniqueness of C^(2,beta) with c >= 0}]
Consider the family of problems indexed by $\theta \in [0, 1]$:
\begin{align} \label{family theta problems} \begin{cases}
	Lu & = \quad f \quad \text{in } \Omega, \\
	u & = \quad 0 \quad \text{on } S, \\
	\jump{u} & = \quad 0 \quad \text{on } \Gamma, \\
	B_\theta u & = \quad g \quad \text{on } \Gamma,
\end{cases} \end{align}
where $B_\theta u = \theta Bu + (1-\theta)B'u$ and
\[
	B'u := -\sum_{s=1}^{n-1}\partial_{x_s}^2 u + u.
\]
We note that $B_1 = B$, $B_0 = B'$, and that
\[
	B_\theta u = -\sum_{s,t=1}^{n-1} \tilde{\mathfrak{a}}^{st} \partial_{x_s}\partial_{x_t}u + \sum_{s=1}^{n-1} \tilde{\mathfrak{b}}^s \, \partial_{x_s}u + \theta\jump{\partial_N u} + \tilde{\mathfrak{c}}u,
\]
where all of the coefficients $\tilde{\mathfrak{a}}^{st}$, $\tilde{\mathfrak{b}}^s$, $\tilde{\mathfrak{c}}$ of $B_\theta$ are bounded in $C^{0,\beta}(\Gamma)$ independently of $\theta$ with $\tilde{\mathfrak{c}} > 0$ and
\[
	\min(1,\mu)|\xi|^2 =: \mu_\theta |\xi|^2 \le \tilde{\mathfrak{a}}^{st} \xi_s \xi_t \quad \text{for all } x \in \Gamma, \: \xi \in \mathbb{R}^{n-1}.
\]
Consider any solution $u \in C^0(\overline{\Omega}) \cap C^{2,\beta}(\overline{\Omega_1}) \cap C^{2,\beta}(\overline{\Omega_2})$ of \eqref{family theta problems}. Then by estimates \eqref{Estimate Theorem A priori Holder} and \eqref{a priori est wrt f and g}, the following inequality holds
\begin{align} \label{est u for method of cont}
	\|u_1\|_{C^{2,\beta}(\Omega_1)} + \|u_2\|_{C^{2,\beta}(\Omega_2)} \le C \left( \|f\|_{C^{0,\beta}(\Omega_1)} + \|f\|_{C^{0,\beta}(\Omega_2)} + \|g\|_{C^{0,\beta}(\Gamma)} \right),
\end{align}
where the constant $C$ is independent of $\theta$. Note that the above estimate is valid for $\Omega_k$ with $k = 1,2$ since $S \cap \Gamma = \emptyset$.

Next, recalling the definition of $X$ as in \eqref{defn X}, let $Y = Y_1 \times Y_2$ where
\begin{align*}
	Y_1 = C^{0,\beta}(\overline{\Omega_1}) \cap C^{0,\beta}(\overline{\Omega_2}), \quad \text{and} \quad Y_2 = C^{0,\beta}(\Gamma).
\end{align*}
Then $Y$ is a Banach space with respect to the norm
\begin{align*}
	\|(f,g)\|_Y &:= \|f\|_{Y_1} + \|g\|_{Y_2} := \|f\|_{C^{0,\beta}(\Omega_1)} + \|f\|_{C^{0,\beta}(\Omega_2)} + \|g\|_{C^{0,\beta}(\Gamma)}.
\end{align*}
Thus, problem \eqref{family theta problems} can be written as
\[
	\mathfrak{L}_\theta u := (Lu, B_\theta u) = (f,g),
\]
where $\mathfrak{L}_\theta: X \to Y$, so the solvability of the problem \eqref{family theta problems} for arbitrary $f \in C^{0,\beta}(\overline{\Omega_1}) \cap C^{0,\beta}(\overline{\Omega_2})$ and $g \in C^{0,\beta}(\Gamma)$ is then equivalent to the invertibility of the mapping $\mathfrak{L}_\theta$. We note that $\mathfrak{L}_0$ and $\mathfrak{L}_1$ are bounded operators.

On the other hand, by Lemma \ref{Lemma showing u in C^0,beta}, Lemma \ref{Lemma Max Principle elliptic theory}, and estimate \eqref{est u for method of cont}, we have
\begin{align*}
	\|u\|_{C^{0,\beta}(\Omega)} &\le C \left( \|u\|_{C^{0,\beta}(\Omega_1)} + \|u\|_{C^{0,\beta}(\Omega_2)} \right) \\
	&\le C_\epsilon \|u\|_{C^0(\Omega)} + \epsilon \left( \|u_1\|_{C^{2,\beta}(\Omega_1)} + \|u_2\|_{C^{2,\beta}(\Omega_2)} \right) \\
	&\le C\left( \|g\|_{C^{0,\beta}(\Gamma)} + \|f\|_{C^{0,\beta}(\Omega_1)} + \|f\|_{C^{0,\beta}(\Omega_2)} \right)
\end{align*}
for some $\epsilon > 0$, and hence
\begin{align*}
	\|u\|_X &= \|u_1\|_{C^{2,\beta}(\Omega_1)} + \|u_2\|_{C^{2,\beta}(\Omega_2)} + \|u\|_{C^{0,\beta}(\Omega)} \\
	&\le C\left( \|f\|_{C^{0,\beta}(\Omega_1)} + \|f\|_{C^{0,\beta}(\Omega_2)} + \|g\|_{C^{0,\beta}(\Gamma)} \right) \\
	&= C\left( \|f\|_{Y_1} + \|g\|_{Y_2} \right) = C \|\mathfrak{L}_\theta u\|_Y,
\end{align*}
where the constant $C$ does not depend on $\theta$. Thus, by the method of continuity (see, for example, \cite[Theorem 5.2]{Gilbarg_Trudinger2001}), the surjectivity of $\mathfrak{L}_1$, which we are investigating, is equivalent to that of $\mathfrak{L}_0$ which is the problem
\begin{align} \label{L_0 problem} \begin{cases}
	Lu & = \quad f \quad \text{in } \Omega, \\
	u & = \quad 0 \quad \text{on } S, \\
	\jump{u} & = \quad 0 \quad \text{on } \Gamma, \\
	B' u & = \quad g \quad \text{on } \Gamma.
\end{cases} \end{align}
Finally, we recall that 
\[
	B'u = -\sum_{s=1}^{n-1} \partial_{x_s}^2 u + u
\]
is invertible on $\Gamma$. If $\varphi \in C^{2,\beta}(\Gamma)$ is the unique solution to $B' \varphi = g$ on $\Gamma$ for a given $g \in C^{0,\beta}(\Gamma)$, then by \cite[Lemma 6.38]{Gilbarg_Trudinger2001} we can make an extension to have $\varphi \in C^{2,\beta}(\overline{\Omega_1}) \cap C^{2,\beta}(\overline{\Omega_2})$. Now we have a Dirichlet problem
\[
	Lu_k = f \quad \text{in } \Omega_k, \qquad u_k = 0 \quad \text{on } S, \qquad u_k = \varphi \quad \text{on } \Gamma,
\]
which has a unique solution $u_k \in C^{2,\beta}(\overline{\Omega_k})$ by \cite[Theorem 6.14]{Gilbarg_Trudinger2001}. Therefore, by Lemma \ref{Lemma showing u in C^0,beta}, we conclude that there is a unique solution in $C^{0,\beta}(\overline{\Omega}) \cap C^{2,\beta}(\overline{\Omega_1}) \cap C^{2,\beta}(\overline{\Omega_2})$ to the system \eqref{Main Eqn}.
\end{proof}

\begin{remark} \label{Remark Fredholm index 0}
As a consequence of Theorem \ref{Theorem Existence and uniqueness of C^(2,beta) with c >= 0}, we see that $\mathfrak{L}_1 = (L, B)$ is a Fredholm operator of index $0$ despite the sign of the transmission term.  Indeed, for $\theta \in [0, 1]$, consider the following linear operator
\[
	\widetilde{\mathfrak{L}}_\theta := \left( Lu, (1-\theta) B + \theta \widetilde{B} \right),
\]
where $\widetilde{\mathfrak{L}}_\theta: X \to Y$ and
\[
	\widetilde{B}u = -\sum_{s,t=1}^{n-1} \partial_{x_s}\left(\mathfrak{a}^{st}\partial_{x_t}u\right) - \jump{\partial_N u} + \sum_{s=1}^{n-1} \mathfrak{b}^s \partial_{x_s}u + \mathfrak{c} u
\]
with coefficients $\mathfrak{a}^{st}$, $\mathfrak{b}^s$, and $\mathfrak{c}$ satisfying the hypotheses of Theorem \ref{Theorem Existence and uniqueness of C^(2,beta) with c >= 0}.  Note that the sign of the transmission term is unfavorable.  It is clear that the map $\theta \mapsto \widetilde{\mathfrak{L}}_\theta \in \mathcal{L}(X, Y)$ is continuous.  Then Schauder estimate from Theorem \ref{Theorem A priori estimate} and Remark \ref{Remark sign jump co-normal term} give
\begin{align*}
	&\|u_1\|_{C^{2,\beta}(\Omega_1)} + \|u_2\|_{C^{2,\beta}(\Omega_2)} + \|u\|_{C^{0,\beta}(\Omega)} \\
	&\le C \| \widetilde{\mathfrak{L}}_\theta u \|_Y + C_\epsilon \|u\|_{C^0(\Omega)} + \epsilon \left( \|u_1\|_{C^{2,\beta}(\Omega_1)} + \|u_2\|_{C^{2,\beta}(\Omega_2)} \right)
\end{align*}
for some small $\epsilon > 0$, so
\[
	(1-\epsilon) \left( \|u_1\|_{C^{2,\beta}(\Omega_1)} + \|u_2\|_{C^{2,\beta}(\Omega_2)} \right) + \|u\|_{C^{0,\beta}(\Omega)} \le C_\epsilon \|u\|_{C^0(\Omega)} + C\| \widetilde{\mathfrak{L}}_\theta u \|_Y.
\]
Choosing $\epsilon>0$ small, we have
\[
	\|u\|_X \le C \left( \|u\|_{C^0(\Omega)} + \|\widetilde{\mathfrak{L}}_\theta u\|_Y \right)
\]
for some constant $C>0$ independent of $\theta$, which implies that $\widetilde{\mathfrak{L}}_\theta$ has finite dimensional null space and closed range.  Thus, $\widetilde{\mathfrak{L}}_\theta$ is semi-Fredholm.  If $\theta < \frac{1}{2}$, the map $\widetilde{\mathfrak{L}}_\theta$ is invertible by Theorem \ref{Theorem Existence and uniqueness of C^(2,beta) with c >= 0} and hence has index $0$.  By the continuity of the index, it also holds for $\theta \ge \frac{1}{2}$, which means that we have Fredholm index $0$ regardless of the sign of the transmission term.  
\end{remark}

\subsection{Fredholm property}
In light of Remark \ref{Remark Fredholm index 0}, it suffices to take $\alpha = +1$. To simplify our notation, we write $\mathfrak{L}$ for $\mathfrak{L}_1$, which is the problem we are considering. If $L$ and $B$ do not satisfy the conditions $c \ge 0$ and $\mathfrak{c} > 0$, it is still possible to assert a Fredholm alternative, which we formulate as in Theorem \ref{Theorem Fredholm solvability}.
\begin{proof}[Proof of Theorem \ref{Theorem Fredholm solvability}]
For all $\sigma, \tau \in \mathbb{R}$, notice that for $u \in X$, $(f, g) \in Y$,
\[
	\mathfrak{L}u = (f,g)
\]
is equivalent to
\[
	\mathfrak{L}_{\sigma,\tau} u = (f + \sigma u, g + \tau u),
\]
where $\mathfrak{L}_{\sigma,\tau} u := \big( (L+\sigma) u, (B + \tau) u \big)$. From Theorem \ref{Theorem Existence and uniqueness of C^(2,beta) with c >= 0}, the mapping $\mathfrak{L}_{\sigma, \tau} u: X \to Y$ is invertible for $\sigma$ and $\tau$ sufficiently large. Now, applying $\mathfrak{L}^{-1}_{\sigma,\tau}$ to both sides, we obtain
\[
	u = \mathfrak{L}_{\sigma, \tau}^{-1}(f + \sigma u, g + \tau u|_\Gamma)
\]
which can be written as
\[
	u - \mathfrak{L}^{-1}_{\sigma, \tau} (\sigma u, \tau u|_\Gamma) = \mathfrak{L}^{-1}_{\sigma, \tau}(f,g).
\]
Letting $\mathcal{K} u: u \in X \subset Y_1 \mapsto \mathfrak{L}^{-1}_{\sigma, \tau}(\sigma u, \tau u|_\Gamma) \in Y_1$, and $h:=\mathfrak{L}^{-1}_{\sigma, \tau}(f,g)$, the equation becomes
\begin{align} \label{I-K eqn}
	(I - \mathcal{K})u = h.
\end{align}
We claim that $\mathcal{K}$ is a compact operator. Let $\{(f_m, g_m)\} \subset Y$ be bounded, and define $u_m := \mathcal{K}(f_m, g_m) \in Y_1$. We want to show that $\{u_m\}$ has a convergent subsequence in $Y_1$. By definition of $u_m$ and $\mathcal{K}$, we have
\[
	\begin{cases} Lu_m + \sigma u_m = f_m &\text{in } \Omega \\ Bu_m + \tau u_m = g_m &\text{on } \Gamma, \end{cases}
\]
where $u_m \in X$, $f_m \in Y_1$, $g_m \in Y_2$. Thus, by Theorem \ref{Theorem A priori estimate}, there exists a positive constant $C=C(n,\beta,L,B,\lambda,\mu)$ such that
\begin{align} \label{Estimate u_m}
	\|u_m\|_{C^{2,\beta}(\Omega_1)} + \|u_m\|_{C^{2,\beta}(\Omega_2)} \le C \big( & \|u_m\|_{C^0(\Omega)} \\ + & \|f_m\|_{C^{0,\beta}(\Omega_1)} + \|f_m\|_{C^{0,\beta}(\Omega_2)}
	+ \|g_m\|_{C^{0,\beta}(\Gamma)} \big). \nonumber
\end{align}
Note that the estimate holds for $\Omega_k$ since $S \cap \Gamma = \emptyset$. Since $C^{0,\beta}(\overline{\Omega}) \subset\subset C^0(\overline{\Omega})$ and $C^{2,\beta}(\overline{\Omega_k}) \subset\subset C^{0,\beta}(\overline{\Omega_k})$, $k = 1, 2$, using estimates as in the proof of Theorem \ref{Theorem Existence and uniqueness of C^(2,beta) with c >= 0}, we find that
\begin{align*}
	\|u_m\|_{C^0(\Omega)} \le C \|u_m\|_{C^{0,\beta}(\Omega)} & \le C \big( \|f_m\|_{C^{0,\beta}(\Omega_1)} + \|f_m\|_{C^{0,\beta}(\Omega_2)} + \|g_m\|_{C^{0,\beta}(\Gamma)} \big).
\end{align*}
Then the inequality \eqref{Estimate u_m} becomes
\[
	\|u_m\|_{C^{2,\beta}(\Omega_1)} + \|u_m\|_{C^{2,\beta}(\Omega_2)} \le C \big( \|f_m\|_{C^{0,\beta}(\Omega_1)} + \|f_m\|_{C^{0,\beta}(\Omega_2)} + \|g_m\|_{C^{0,\beta}(\Gamma)} \big),
\]
or we can write this to be
\[
	\|u_m\|_{C^{2,\beta}(\Omega_1)} + \|u_m\|_{C^{2,\beta}(\Omega_2)} + \|u_m\|_{C^{0,\beta}(\Omega)} \le C \big( \|f_m\|_{C^{0,\beta}(\Omega_1)} + \|f_m\|_{C^{0,\beta}(\Omega_2)} + \|g_m\|_{C^{0,\beta}(\Gamma)} \big),
\]
which is equivalent to
\[
	\|u_m\|_X \le C \|(f_m, g_m)\|_Y,
\]
so $\|u_m\|_X$ is bounded in $X$. Since $X \subset\subset Y_1$, we conclude that $\{u_m\}$ contains a subsequence $\{u_{m_k}\}$ such that $u_{m_k} \to u$ in $Y_1$, which proves the claim that $\mathcal{K}$ is a compact operator. \\
Applying the Fredholm Alternative, equation \eqref{I-K eqn} always has a solution $u\in X$ provided the homogeneous equation $(I-\mathcal{K})u=0$ has only the trivial solution $u=0$. When this condition is not satisfied, the kernel of $I-\mathcal{K}$ is a finite dimensional subspace of $Y_1$. Since the solutions of (\ref{I-K eqn}) are in one-to-one correspondence to the solutions of (\ref{Main Eqn}), we therefore can conclude the alternative stated in the theorem.
\end{proof}

Finally, the last result in this subsection gives H\"older continuity for a classical solution provided sufficient smoothness of the data and coefficients.  

\begin{prop} \label{Theorem elliptic higher regularity}
Suppose the coefficients $a^{ij}, b^i, c, \mathfrak{a}^{st}, \mathfrak{b}^s, \mathfrak{c}$ exhibit the same regularity as in Theorem \ref{Theorem A priori estimate}. If $u \in C^0(\overline\Omega) \cap C^2(\overline{\Omega_1}) \cap C^2(\overline{\Omega_2})$ is a solution to equation \eqref{Main Eqn} with $\alpha = \pm 1$ for $f \in C^{0, \beta}(\overline{\Omega_1}) \cap C^{0, \beta}(\overline{\Omega_2})$ and $g \in C^{0, \beta}(\Gamma)$, then $u \in C^{0,\beta}(\overline{\Omega}) \cap C^{2,\beta}(\overline{\Omega_1}) \cap C^{2,\beta}(\overline{\Omega_2})$.
\end{prop}

\begin{proof}
By the hypothesis, we have $u_k \in C^2(\Gamma)$, and hence, $\partial_N u_k \in C^1(\Gamma) \subset\subset C^{0, \beta}(\Gamma)$ for $k = 1,2$. Thus, the boundary condition $Bu = g$ can be re-expressed as
\[
	-\sum_{s,t=1}^{n-1} \partial_{x_s}\left( \mathfrak{a}^{st} \partial_{x_t} u \right) + \sum_{s=1}^{n-1} \mathfrak{b}^s \partial_{x_s}u + \mathfrak{c} u = h \quad \text{on } \Gamma,
\]
where $h := g - \alpha \jump{\partial_N u} \in C^{0, \beta}(\Gamma)$. By standard elliptic regularity theory, $u|_\Gamma \in C^{2, \beta}(\Gamma)$. Now the Dirichlet problem
\[
	Lu_k = f \quad \text{in } \Omega_k, \qquad u_k = 0 \quad \text{on } S, \qquad u_k = u|_\Gamma \quad \text{on } \Gamma
\]
has a unique solution $u \in C^{2, \beta}(\overline{\Omega_1}) \cap C^{2, \beta}(\overline{\Omega_2})$. Using the fact that $C^{2, \beta}(\overline{\Omega_k}) \subset \subset C^{0, \beta}(\overline{\Omega_k})$ and Lemma \ref{Lemma showing u in C^0,beta}, we conclude that $u \in C^{2, \beta}(\overline{\Omega_1}) \cap C^{2, \beta}(\overline{\Omega_2}) \cap C^{0, \beta}(\Omega)$.
\end{proof}

\begin{remark} \label{Remark sign of B switched}
	If we change the boundary term $B$ to
	\[
	\widehat{B}u := \sum_{s,t=1}^{n-1} \partial_{x_s}\left(\mathfrak{a}^{st}\partial_{x_t}u\right) + \alpha \jump{\partial_N u} + \sum_{s=1}^{n-1} \mathfrak{b}^s \partial_{x_s}u + \mathfrak{c} u,
	\]
	where the signs of the second-order term is switched, we obtain the same results as in Theorems \ref{Theorem A priori estimate}, \ref{Theorem Fredholm solvability}, and Proposition \ref{Theorem elliptic higher regularity}.  For Lemma \ref{Lemma Max Principle elliptic theory} and Theorem \ref{Theorem Existence and uniqueness of C^(2,beta) with c >= 0} to be valid, we have to assume in addition that $\alpha = -1$ and $\mathfrak{c} < 0$, which means the signs of the second-order term and zeroth-order term must be opposite.  
\end{remark}

\section{Steady capillary-gravity waves in the presence of wind} \label{existence section}
In this section, we will apply the results found above to investigate the existence of steady wind-driven water waves. There exists a well-known change of variables due to Dubreil-Jacotin that maps $\Omega$ to a strip (see \cite{Dubreil-Jacotin1934}). We change variables $(x,y) \in \Omega \mapsto (x, -\psi) =: (q,p) \in D$. We recall that $\psi$ is the (relative) pseudostream function for the flow defined by \eqref{pseudostream function defn}, along with the boundary conditions $\psi=0$ on the upper lid, $\psi=-p_0$ at the bed, and $\psi \in C^{0,\alpha}(\overline{\Omega}) \cap C^{2,\alpha}(\overline{\Omega_1}) \cap C^{2,\alpha}(\overline{\Omega_2})$ for a fixed $\alpha \in (0,1)$. Thus, the problem is now posed in a union of rectangles $D = D_1 \cup D_2 \subset \mathbb{R}^2$, where the air region is mapped to
\[
	D_1 := \{(q,p) \in D: 0 < q < 2\pi, p_1 < p < 0 \},
\]
and the water region is mapped to
\[
	D_2 := \{(q,p) \in D: 0 < q < 2\pi, p_0 < p < p_1 \}.
\]
With that in mind, we have definitions for the lid, the free surface, and the ocean bed respectively as follows
\[
	T := \{p=0\}, \qquad I := \{p=p_1\}, \qquad B := \{p=p_0\}.
\]
Under this change of coordinates, the Euler problem \eqref{Eulerian eqn 1}$-$\eqref{wo stagnation} becomes the following \emph{height equation}
\begin{equation} \label{height equation air vorticity}
\begin{dcases}
	(1 + h_q^2)h_{pp} + h_{qq}h^2_p - 2h_p h_q h_{pq} = -\gamma(-p) h_p^3 & \text{in } D_1, \\
	(1 + h_q^2)h_{pp} + h_{qq}h^2_p - 2h_p h_q h_{pq} = 0 & \text{in } D_2, \\
	\dsp \jump{\frac{1 + h_q^2}{h_p^2}} + 2g\jump{\rho}h - Q + \sigma \frac{h_{qq}}{(1 + h_q^2)^{3/2}} = 0 & \text{on } p = p_1, \\
	\jump{h} = 0 & \text{on } p = p_1, \\
	h = 0 & \text{on } p = p_0, \\
	h = \ell + d(h) & \text{on } p = 0,
\end{dcases}
\end{equation}
where $h(q,p)$ is the height above the bed of the point $(x,y)$, where $x = q$ and $(x,y)$ lies on $\{-\psi = p\}$, and the depth operator $d$ is defined to be
\[
	d(h) := \frac{1}{2\pi} \int_{-\pi}^{\pi} h(q, p_1) \,\mathrm{d}q.
\]
Note that $\rho$ in the above equation is for $(q,p)$-coordinates after the transformation. The equivalence of \eqref{height equation air vorticity} to the original system \eqref{Eulerian eqn 1}--\eqref{wo stagnation} can be proved following \cite[Lemma $A.2$]{Chen_Walsh2016}.

Our objective is to find solutions $(h,Q) \in \solnSpace$, where
\[
	\solnSpace := \left( C^{2,\alpha}_{\text{per}} \left(\overline{D_1}\right) \cap C^{2,\alpha}_{\text{per}} \left(\overline{D_2}\right) \cap C^{0,\alpha}_{\text{per}} \left(\overline{D}\right) \right) \times \mathbb{R}
\]
and $h_p>0$ in $\overline D$ because of no stagnation condition \eqref{wo stagnation}. Recall that the space $C^{k,\alpha}_\mathrm{per}(\overline R)$ is the set of $C^{k,\alpha}(\overline R)$ functions that are $2\pi$-periodic and even in their first coordinate. The presence of surface tension $\sigma$ is manifested as the nonlinear second-order term in the boundary condition.

We will prove the following theorem stated in the Dubreil-Jacotin variables, which implies Theorem \ref{Theorem existence of small amplitude wind-drive water waves}.

\begin{thm}[existence] \label{Theorem local bifurcation air vorticity surface tension}
Let $p_1 < 0$, $\ell > 0$, and atmospheric vorticity function $\gamma \in C^{0,\alpha}((p_1,0))$ be given. Then there exists $\sigma_0 \ge 0$ such that for each $\sigma > \sigma_0$, there is a continuous curve $\mathcal{C}_\mathrm{loc} \subset \solnSpace$ of solution to \eqref{height equation air vorticity} with the following properties:
\begin{enumerate}[label=\roman*)]
	\item $\mathcal{C}_\mathrm{loc} := \{ \left(h(\lambda),Q(\lambda) \right) : \lambda \in (\lambda^* + \epsilon, \lambda^* - \epsilon) \}$, \\ where $\lambda \in (\lambda^* + \epsilon, \lambda^* - \epsilon) \mapsto \left(h(\lambda),Q(\lambda) \right) \in \solnSpace$ is $C^1$.
	\item $(h(\lambda^*), Q(\lambda^*)) = (H(\lambda^*), Q(\lambda^*))$ is a laminar solution.
	\item $h(\lambda)$ is non-laminar for $\lambda \ne \lambda^*$.
\end{enumerate}
\end{thm}

\begin{remark}
In fact, there is a necessary and sufficient condition that we call \emph{local bifurcation condition} \eqref{LBC air vorticity}, which will be given explicit in Lemma \ref{Lemma size condition air vorticity}. In particular, \eqref{LBC air vorticity} always holds for $\sigma$ sufficiently large. When $\sigma$ is small, a local bifurcation argument can still be carried out, but the eigenvalue of the linearized problem may not be simple. In this case, a more sophisticated analysis is required (see, for example, \cite{Wahlen2006,Wahlen_gravity2006,Walsh2014}).
\end{remark}

\subsection{Laminar solutions}
We first consider \emph{laminar flows} which are solutions of the height equation \eqref{height equation air vorticity} that are independent of $q$. Physically, this entails a wave where all of the streamlines are parallel to the bed. These will serve as the trivial solution curve when we apply the Crandall-Rabinowitz theorem to obtain Theorem \ref{Theorem local bifurcation air vorticity surface tension}.

Let us define $\GammaRel$ by
\begin{equation} \label{Gamma_rel definition air vorticity}
	\partial_p(\GammaRel(p)^2) = 2\gamma(-p), \qquad \ell = \int_{p_1}^0 \frac{\mathrm{d}p}{\GammaRel(p)}.
\end{equation}

\begin{remark} \label{Remark Gamma_Rel}
$\GammaRel$ is called the \emph{(pseudo) relative circulation} and is given by
\[
	\GammaRel(p) = \frac{1}{2\pi} \int_{\{ \psi=-p \}} |\nabla \psi| \,\mathrm{d}\mathcal{H}^1,
\]
where $\mathcal{H}^1$ denotes one-dimensional Hausdorff measure.  Note that circulation around a closed loop is conserved for the time-dependent problem by Kelvin's circulation law.  For periodic domains, this includes the circulation along the streamlines $\{\psi=-p\}$.  If the waves we construct are to be viewed as generated dynamically by the wind, the circulation along each streamline must agree with the initial configuration.  
\end{remark}

For laminar flows, since $h$ does not depend on $q$, we can write $h = H(p)$, where $H$ satisfies the following ODE:
\begin{equation} \label{laminar equation air vorticity}
	\begin{dcases}
		H_{pp} = -\gamma(-p) H_p^3 & \text{in } p_1 < p < 0, \\
		H_{pp} = 0 & \text{in } p_0 < p < p_1, \\
		\jump{H_p^{-2}} + 2g \jump{\rho}H - Q = 0 & \text{on } p = p_1, \\
		H = 0 & \text{on } p = p_0, \\
		H = \ell + d(H) & \text{on } p = 0.
	\end{dcases}
\end{equation}
Note that $d(H) = H(p_1)$. The above equation can be solved explicitly, but we still need some compatibility conditions to ensure continuity across the interface.

\begin{lemma}[laminar flow] \label{Lemma laminar solution air vorticity}
If the compatibility condition (\ref{Gamma_rel definition air vorticity}) is satisfied, then there exists a one-parameter family of solutions $\{(H(\cdot;\lambda), Q(\lambda)) : \lambda > 0 \}$ to the laminar flow equation (\ref{laminar equation air vorticity}) with $H_p > 0$. Each member of the family has the explicit form
\begin{equation} \label{laminar solution air vorticity}
	H(p;\lambda) = \begin{dcases}
	\dsp \int_{p_1}^p \frac{\mathrm{d}s}{\GammaRel(s)} + \frac{p_1 - p_0}{\lambda}, & p_1 < p < 0, \\
	\frac{p - p_0}{\lambda}, & p_0 < p < p_1,
	\end{dcases}
\end{equation}
and
\begin{equation} \label{Q(lambda) laminar solution air vorticity}
	Q(\lambda) = \frac{2g \jump{\rho}(p_1 - p_0)}{\lambda} + \GammaRel(p_1)^2 - \lambda^2.
\end{equation}
Moreover, the depth of the fluid at parameter value $\lambda$ is
\begin{equation} \label{d(lambda) air vorticity}
	d(H(\cdot;\lambda)) = \frac{p_1 - p_0}{\lambda}.
\end{equation}
\end{lemma}


Since the laminar flow is independent of the surface tension $\sigma$, the proof of Lemma \ref{Lemma laminar solution air vorticity} can be obtained by similar arguments as in \cite[Lemma 4.2]{Buhler_Shatah_Walsh2013}, which we will omit. Note that differentiating (\ref{Q(lambda) laminar solution air vorticity}) with respect to $\lambda$ gives
\[
	Q'(\lambda) = -\frac{2g\jump{\rho}(p_1 - p_0)}{\lambda^2} - 2\lambda
\]
and
\[
	Q''(\lambda) = \frac{4g\jump{\rho}(p_1-p_0)}{\lambda^3} - 2 < 0,
\]
so $\lambda \mapsto Q(\lambda)$ is concave and has a unique maximum at $\lambda_0$ satisfying
\begin{equation} \label{lambda_0 formula}
	\lambda_0^3 = -g\jump{\rho}(p_1-p_0).
\end{equation}

\subsection{Linearized problem}
Next, let us consider the linearization of the height equation \eqref{height equation air vorticity} at one of the laminar solutions $(H(\cdot;\lambda), Q(\lambda))$ constructed in Lemma \ref{Lemma laminar solution air vorticity}:
\begin{equation} \label{linearized equation air vorticity}
	\begin{dcases}
		(a^3 m_p)_p + (a m_q)_q = 0 & \text{in } D_1 \cup D_2, \\
		-2\jump{a^3 m_p} + 2g\jump{\rho}m + \sigma m_{qq} = 0 & \text{on } p=p_1, \\
		m = 0 & \text{on } p = p_0, \\
		m - d(m) = 0 & \text{on } p = 0,
	\end{dcases}
\end{equation}
where
\[
	a := a(p;\lambda) = H_p(p;\lambda)^{-1} = \begin{dcases}
	\GammaRel(p), & p_1 < p < 0, \\
	\lambda, & p_0 < p < p_1.
	\end{dcases}
\]

Since we seek solutions that are $2\pi$-periodic and even in $q$, we first consider $m$ of the form $m(q,p) = M(p)\cos(nq)$, for some $n \ge 0$. If $n = 0$, $m$ does not depend on $q$ and the linearized problem \eqref{linearized equation air vorticity} becomes
\begin{equation*}
	\begin{dcases}
		(a^3 M_p)_p = 0 & \text{in } p_1 < p < 0, \\
		M_{pp} = 0 & \text{in } p_0 < p < p_1, \\
		-\jump{a^3 M_p} + g\jump{\rho}M = 0 & \text{on } p = p_1, \\
		M = 0 & \text{on } p = p_0, \\
		M - d(M) = 0 & \text{on } p = 0.
	\end{dcases}
\end{equation*}
This equation can be solved explicitly. Using the boundary condition at $p = p_0$ and the continuity of $M$ across the interface, we find that in the water region
\[
	M^{(2)}(p) = \frac{p - p_0}{p_1 - p_0} M(p_1), \quad \text{in } p_0 < p < p_1.
\]
For the air region, we first observe that
\[
	M(0) = d(M) = \frac{1}{2\pi} \int_{-\pi}^{\pi} M(p_1) \,\mathrm{d}q = M(p_1),
\]
and hence, $M_p$ must vanish at least once inside $(p_1,0)$ by Rolle's Theorem. We also have, from the above ODE, that $a^3 M_p$ is constant, so we conclude that $M_p \equiv 0$ in $(p_1,0)$. Finally, the jump condition gives
\[
	-\frac{\lambda^3}{p_1 - p_0} M(p_1) = g\jump{\rho} M(p_1),
\]
which implies that there can be a zero-mode solution if and only if $\lambda = \lambda_0$, where $\lambda_0$ is defined according to \eqref{lambda_0 formula}.

On the other hand, if $n > 0$, the linearized problem \eqref{linearized equation air vorticity} becomes
\begin{equation} \label{Pn air vorticity}
	\begin{dcases}
		-(a^3 M_p)_p = -n^2 a M & \text{in } (p_0,p_1) \cup (p_1,0), \\
		2\jump{a^3 M_p} - 2g \jump{\rho} M = -n^2 \sigma M & \text{on } p = p_1, \\
		M = 0 & \text{on } p = p_0, \\
		M = 0 & \text{on } p = 0.
	\end{dcases}
\end{equation}
To investigate the ODE (\ref{Pn air vorticity}), we consider the following general eigenvalue problem
\begin{equation} \label{P_mu eigenvalue problem}
	\tag{$P_\mu$}
	\begin{dcases}
		-\frac{1}{a}(a^3 u')' = \mu u & \text{in } (p_0, p_1) \cup (p_1, 0), \\
		2\jump{a^3 u'} - 2g \jump{\rho} u = \mu \sigma u & \text{on } p = p_1, \\
		u = 0 & \text{on } p = p_0, \\
		u = 0 & \text{on } p = 0.
	\end{dcases}
\end{equation}
In particular, we are interested in the case $\mu = -n^2$.

The eigenvalue problem \eqref{P_mu eigenvalue problem} closely resembles a Sturm--Liouville equation, but the eigenvalue occurs both in the interior and boundary conditions. Moreover, the associated inner product defining the relation between eigenfunctions is indefinite. For that reason, it is natural to reformulate it in a Pontryagin space. Here we follow the general approach of Wahl\'en \cite{Wahlen_gravity2006,Wahlen2006} and Walsh \cite{Walsh2014}.

With that in mind, we introduce the complex Pontryagin space (see \cite{Bognar1974,Iohvidov1982})
\[
	\mathbb{H} := \left\{ \tilde{u} = (u,b) \in L^2([p_0, 0]) \times \mathbb{C} \right\}
\]
with the indefinite inner product
\[
	[\tilde{u}_1, \tilde{u}_2] := \langle au_1, u_2 \rangle_{L^2} - \frac{1}{2\sigma} b_1 \overline{b_2}.
\]
We understand that the $L^2$-inner product is taken over $(p_0, 0)$. On $\mathbb{H}$, there is also an associated Hilbert space inner product, given by $\langle \tilde{u}, \tilde{v} \rangle_\mathbb{H} = [J\tilde{u}, \tilde{v}]$, where
\[
	J = \left( \begin{matrix}
		I & 0 \\
		0 & -1
	\end{matrix} \right).
\]

\begin{prop} \label{Proposition maximal negative definite subspace dimension 1}
$\mathbb{H}$ is $\pi_1$-space, that is $\mathbb{H} = \mathbb{H}_+ \oplus \mathbb{H}_-$, where
\begin{align*}
	\mathbb{H}_+ & \subset \{ x \in \mathbb{H}: [x,x] > 0, \text{ or } x=0 \}, \\
	\mathbb{H}_- & \subset \{ x \in \mathbb{H}: [x,x] < 0, \text{ or } x=0 \}
\end{align*}
are complete subspaces with $\dim\mathbb{H}_+=1$ or $\dim\mathbb{H}_-=1$.
\end{prop}
We omit the proof of this proposition, as it is elementary.  In fact, we have explicitly that $\mathbb{H} = \mathbb{H}_+ \oplus \mathbb{H}_-$, where
\begin{align*}
	\mathbb{H}_+ & := L^2((p_0, 0)) \times \{0\}, \\
	\mathbb{H}_- & := \left\{ 0 \in L^2((p_0, 0) \right\} \times \mathbb{C}.
\end{align*}


Next, define the linear operator $K: D(K) \subset \mathbb{H} \to \mathbb{H}$ by
\[
	K\tilde{u} := \left( -\frac{1}{a} \left( a^3 u' \right)' \, , \, 2\jump{a^3 u'} - 2g\jump{\rho} u(p_1) \right),
\]
where
\begin{multline*}
	D(K) := \big\{ \tilde{u} = (u, b) \in \big( H^2((p_0,p_1)) \cap H^2((p_1,0)) \cap C^0((p_0,0)) \big) \times \mathbb{C}: \\ u(p_0) = u(0) = 0, \sigma u(p_1) = b \big\}.
\end{multline*}
Thus, there exists a nontrivial solution of \eqref{P_mu eigenvalue problem} if and only if $\mu$ is an eigenvalue of $K$. Moreover, it is clear that $D(K)$ is dense in $\mathbb{H}$, and the operator $K$ is closed. Recalling the convention $\jump{u} = u|_{p_1^+} - u|_{p_1^-}$ and using integration by parts, we can show
\begin{align} \label{K symmetric}
	[K\tilde{u}, \tilde{u}] = \left\langle a^3 u', u' \right\rangle_{L^2} + g\jump{\rho} |u(p_1)|^2 = [\tilde{u}, K\tilde{u}] \in \mathbb{R},
\end{align}
which implies that $K$ is symmetric, and in fact, self-adjoint. The next proposition provides a condition under which the operator $K$ is positive, that is, $[K\tilde{u}, \tilde{u}] > 0$ for all non-zero $\tilde{u} \in D(K)$.

\begin{prop} \label{Proposition unique eigenvalue for operator K}
$K$ is self-adjoint with simple eigenvalues. Moreover, it has a maximal negative semidefinite subspace invariant under $K$ that has dimension one\footnote{We recall that the subspace $\mathcal{M}$ of $\mathbb{H}$ is said to be negative semidefinite, provided that for all $u \in \mathcal{M}$, $[u,u] \le 0$. Moreover, $\mathcal{M}$ is an invariant subspace under the operator $K$ if $K(\mathcal{M}) \subset \mathcal{M}$. Also, $\mathcal{M}$ is said to be a maximal subspace in $\mathbb{H}$ if $\mathcal{M}$ is not contained in any other proper subspaces of $\mathbb{H}$.}. For $\lambda > \lambda_0$, the operator $K$ is positive with a unique negative eigenvalue.
\end{prop}

\begin{proof}
It follows from the above discussion that $K$ is self-adjoint. Since $\mathbb{H}$ is a $\pi_1$-space, \cite[Theorem 12.1']{Iohvidov1982} implies that $K$ has a maximal negative semidefinite subspace invariant under $K$ that has dimension one. By an argument similar to \cite[Lemma 3.8]{Wahlen_gravity2006} and \cite[Lemma 2]{Wahlen2006}, we see that $K$ has discrete spectrum and its eigenvalues are geometrically simple.

Next, since $K$ has a maximal invariant negative semidefinite subspace which is of dimension one, it has at least one eigenvalue of negative-semidefinite type. By this, we mean the restriction of $[\cdot,\cdot]$ to the eigenspace corresponding to an eigenvalue is a negative semidefinite inner product. We caution that this does not say anything about the sign of the eigenvalue itself. Let $\mu$ be a general eigenvalue of $K$ with corresponding non-zero eigenvector $\tilde{u}$.  Taking the complex conjugate of the equation $K\tilde{u} = \mu \tilde{u}$, we see that $\bar{\mu}$ is also an eigenvalue of $K$ (note that the coefficients of the operator $K$ are real). Thus, either $\mu$ is real, or both $\mu$ and its complex conjugate are eigenvalues. For the latter case, the corresponding eigenvector $\tilde{u}$ must be neutral, that is, $[\tilde{u}, \tilde{u}] = 0$. This follows from the observation that
\[
	\mu [\tilde{u}, \tilde{u}] = [K\tilde{u}, \tilde{u}] = [\tilde{u}, K\tilde{u}] = \bar{\mu} [\tilde{u}, \tilde{u}].
\]

On the other hand, if $\mu \in \mathbb{R}$ is an eigenvalue with corresponding eigenvector $\tilde{u}$ such that $[\tilde{u}, \tilde{u}] \ne 0$, then letting
\[
	\mathcal{N} := \text{span } \tilde{u} = \ker(K - \mu I) \quad \text{and} \quad \mathcal{N}^{[\perp]} := \{\tilde{v} \in \mathbb{H}: [\tilde{v}, \mathcal{N}] = 0\},
\]
we have $\mathbb{H} = \mathcal{N} \, [\dot{+}] \, \mathcal{N}^{[\perp]}$ as an orthogonal direct sum. Note that we are using $[\perp]$ and $[\dot{+}]$ to emphasize that the orthogonality is with respect to the indefinite inner product $[\cdot,\cdot]$. If $\tilde{w}$ is in the range of $K - \mu I$, then $\tilde{w} = (K - \mu I)\tilde{v}$ for some $\tilde{v} \in \mathbb{H}$, and hence
\[
	[\tilde{w}, \tilde{u}] = [K\tilde{v}, \tilde{u}] - \mu [\tilde{v}, \tilde{u}] = [\tilde{v}, K\tilde{u}] - [\tilde{v}, \mu \tilde{u}] = 0,
\]
which implies that the range of $K - \mu I$ is in $\mathcal{N}^{[\perp]}$. Thus, $\mu$ is algebraically simple if it is real.  

Finally, suppose $\lambda > \lambda_0$. By the Cauchy--Schwarz inequality,
\begin{align*}
	|u(p_1)|^2 = \left| \int_{p_0}^{p_1} u'(p) \,\mathrm{d}p \right|^2 \le \int_{p_0}^{p_1} a^3 \left|u'\right|^2 \,\mathrm{d}p \int_{p_0}^{p_1} a^{-3} \,\mathrm{d}p < -\frac{1}{g\jump{\rho}} \int_{p_0}^{p_1} a^3 \left|u'\right|^2 \,\mathrm{d}p,
\end{align*}
which gives
\[
	g\jump{\rho} |u(p_1)|^2 + \int_{p_0}^{0} a^3 \left|u'\right|^2 \,\mathrm{d}p > 0.
\]
Thus, $[K\tilde{u}, \tilde{u}] > 0$, that is, $K$ is positive. Then $[\tilde{u}, \tilde{u}] \ne 0$, so $\tilde{u}$ is non-neutral. Hence all eigenvalues are real when $\lambda > \lambda_0$.

If $\mu$ is a negative semidefinite eigenvalue with the corresponding eigenvector $\tilde{u}$, then
\[
	\mu [\tilde{u}, \tilde{u}] = [K\tilde{u}, \tilde{u}] > 0,
\]
and hence, it follows that $\mu < 0$. This means that any real negative semidefinite eigenvalue of $K$ must be negative. In fact, there is only one such eigenvalue. Indeed, if $\nu$ is another eigenvalue with corresponding eigenvector $\tilde{v}$, then
\[
	[\tilde{u}, \tilde{v}] = \frac{1}{\mu} [\mu \tilde{u}, \tilde{v}] = \frac{1}{\mu} [K \tilde{u}, \tilde{v}] = \frac{1}{\mu} [\tilde{u}, K\tilde{v}] = \frac{\nu}{\mu} [\tilde{u}, \tilde{v}],
\]
which implies $\mu = \nu$ or $[\tilde{u}, \tilde{v}] = 0$. Since any maximal invariant semidefinite subspace of $K$ is one dimensional, we must have $\mu = \nu$. We have therefore shown $K$ has a unique negative eigenvalue.
\end{proof}

Define the Rayleigh quotient $\mathscr{R}$ corresponding to \eqref{P_mu eigenvalue problem} by
\[
	\mathscr{R}(\varphi; \lambda) := \frac{\int_{p_0}^0 a^3 \varphi_p^2 \,\mathrm{d}p + g\jump{\rho} \varphi(p_1)^2}{\int_{p_0}^0 a \varphi^2 \,\mathrm{d}p - \frac{\sigma}{2}\varphi(p_1)^2}, \quad \lambda > \lambda_0, \varphi \in \mathscr{A},
\]
where the admissible set is defined by
\begin{multline*}
	\mathscr{A} :=  \Big\{ \varphi \in H^2((p_0,p_1)) \cap H^2((p_1,0)) \cap C((p_0,0)) : \\ \varphi(p_0) = \varphi(0) = 0 \text{ and } \int_{p_0}^0 a\varphi^2 \,\mathrm{d}p - \frac{\sigma}{2} \varphi(p_1)^2 < 0 \Big\}.
\end{multline*}
Note that we are considering $\varphi$ only in the negative definite subspace of $K$ because of the condition
\begin{equation} \label{negative semidefinite}
	\int_{p_0}^0 a\varphi^2 \,\mathrm{d}p - \frac{\sigma}{2} \varphi(p_1)^2 < 0.
\end{equation}

Simple arguments can show that if for a fixed $\lambda > \lambda_0$, $\varphi$ is a critical point of $\mathscr{R}(\cdot; \lambda)$, then $\varphi$ solves \eqref{P_mu eigenvalue problem} for $\mu = \mathscr{R}(\varphi;\lambda)$.

Next, let us define
\[
	\nu(\lambda) := \sup_{\substack{\varphi \in \mathscr{A} \\ \varphi \not\equiv 0}} \mathscr{R}(\varphi; \lambda).
\]
First, we want to show that $-1$ is in the range of $\nu$.  This is because we want our solutions to be $2\pi$-periodic in $q$, and the null space of $\mathcal{F}_m(\lambda^*, 0)$ is spanned by $\varphi_1(p) \cos(q)$ (see Lemma \ref{Lemma null space air vorticity}), which is the case where $n=1$ and hence $\mu = -n^2 = -1$ in \eqref{Pn air vorticity}.

\begin{lemma} \label{Lemma Rayleigh quotient < -n^2}
Let $a_\mathrm{min} := \min_{[p_1, 0]} a$ (which does not depend on $\lambda$). Then for each $n \ge 1$, $\nu(\lambda) < -n^2$ when $\lambda$ satisfies
\[
	\lambda^2 > -a_\mathrm{min}^2 - \frac{g\jump{\rho}}{n} + \frac{\sigma n}{2}.
\]
\end{lemma}

\begin{proof}
Let $\varphi \in \mathscr{A}$ be given and fix any $\lambda$ as in the hypothesis. Then
\begin{align*}
	\int_{p_1}^0 \left( a^3 \varphi_p^2 + n^2 a \varphi^2 \right) \,\mathrm{d}p &\ge a_\mathrm{min} \int_{p_1}^0 \left( a_\text{min}^2 \varphi_p^2 + n^2 \varphi^2 \right) \,\mathrm{d}p \\
	&\ge -2 n a_\mathrm{min}^2 \int_{p_1}^0 \varphi_p \varphi \,\mathrm{d}p \\
	&= -n a_\text{min}^2 \int_{p_1}^0 \left( \varphi^2 \right)_p \,\mathrm{d}p = n a_\mathrm{min}^2 \varphi(p_1)^2.
\end{align*}
On the other hand, since $a^{(2)} = \lambda$,
\[
	\int_{p_0}^{p_1} \left( a^3 \varphi_p^2 + n^2 a \varphi^2 \right) \,\mathrm{d}p = \lambda \int_{p_0}^{p_1} \left( a^2 \varphi_p^2 + n^2 \varphi^2 \right) \,\mathrm{d}p \ge 2 n\lambda^2 \int_{p_0}^{p_1} \varphi_p \varphi \,\mathrm{d}p = n\lambda^2 \varphi(p_1)^2.
\]
Summing these together and using the hypothesis for $\lambda$, we find
\begin{align*}
	\int_{p_0}^0 \left( a^3 \varphi_p^2 + n^2 a\varphi^2 \right) \,\mathrm{d}p &\ge \left( n\lambda^2 + na_\mathrm{min}^2 \right) \varphi(p_1)^2 \\
	&> \left( -g\jump{\rho} + \frac{\sigma n^2}{2}\right) \varphi(p_1)^2,
\end{align*}
which implies
\[
	\int_{p_0}^0 a^3 \varphi_p^2 \,\mathrm{d}p + g\jump{\rho}\varphi(p_1)^2 > -n^2 \left( \int_{p_0}^0 a \varphi^2 \,\mathrm{d}p - \frac{\sigma}{2}\varphi(p_1)^2 \right),
\]
so $\mathscr{R}(\varphi; \lambda) < -n^2$. Thus, $\nu(\lambda) < -n^2$.
\end{proof}
Next, we need to verify that $\nu(\lambda) > -1$ for some $\lambda > \lambda_0$. Since this is not true in general, we will have it as one of our hypotheses.  

\begin{definition} \label{Definition LBC air vorticity}
We say that the \emph{local bifurcation condition} is satisfied provided that
\begin{align} \label{LBC air vorticity}
	\tag{LBC}
	\sup_{\lambda > \lambda_0} \nu(\lambda) > -1.
\end{align}
\end{definition}

This is necessary and sufficient for our main result Theorem \ref{Theorem local bifurcation air vorticity surface tension} to hold. An explicit but not sharp condition is the following:

\begin{lemma}[size condition] \label{Lemma size condition air vorticity}
For
\begin{align} \label{size condition air vorticity}
	\sigma > \frac{2\lambda_0 (p_1 - p_0)}{3} + \frac{2}{p_1^2} \int_{p_1}^0 \left( \GammaRel^3 + p^2 \GammaRel \right) \,\mathrm{d}p,
\end{align}
where $\lambda_0$ is defined as in (\ref{lambda_0 formula}), (\ref{LBC air vorticity}) holds.
\end{lemma}

\begin{proof}
Let
\[
	\varphi(p) := \begin{dcases}
		\frac{p}{p_1}, & p_1 < p < 0, \\
		\frac{p - p_0}{p_1 - p_0}, & p_0 < p < p_1.
	\end{dcases}
\]
We first check if $\varphi$ is in the admissible set $\mathscr{A}$. We see that
\begin{align*}
	\int_{p_0}^0 a\varphi^2 \,\mathrm{d}p - \frac{\sigma}{2} \varphi(p_1)^2 & = \int_{p_0}^{p_1} \lambda \left( \frac{p-p_0}{p_1-p_0} \right)^2 \,\mathrm{d}p + \int_{p_1}^0 \GammaRel \left( \frac{p}{p_1} \right)^2 \,\mathrm{d}p - \frac{\sigma}{2} \\
	&= \frac{\lambda(p_1 - p_0)}{3} + \frac{1}{p_1^2}\int_{p_1}^0 p^2 \GammaRel \,\mathrm{d}p - \frac{\sigma}{2},
\end{align*}
but from the hypothesis \eqref{size condition air vorticity},
\[
	\frac{\lambda_0(p_1 - p_0)}{3} + \frac{1}{p_1^2}\int_{p_1}^0 p^2 \GammaRel \,\mathrm{d}p - \frac{\sigma}{2} < 0.
\]
Thus, for $|\lambda - \lambda_0|$ small, $\varphi \in \mathscr{A}$. With this particular $\varphi$, we then compute
\begin{align*}
	\mathscr{R}(\varphi; \lambda) &= \frac{\dsp \int_{p_0}^{p_1} \frac{\lambda^3}{(p_1 - p_0)^2} \,\mathrm{d}p + \int_{p_1}^0 \frac{\GammaRel^3}{p_1^2} \,\mathrm{d}p + g\jump{\rho}}{\dsp \int_{p_0}^{p_1} \lambda \left( \frac{p - p_0}{p_1 - p_0} \right)^2 \,\mathrm{d}p + \int_{p_1}^0 \GammaRel \left( \frac{p}{p_1} \right)^2 \,\mathrm{d}p - \frac{\sigma}{2}} \\
	&= \frac{\dsp \frac{\lambda^3}{p_1-p_0} + \frac{1}{p_1^2}\int_{p_1}^0 \GammaRel^3 \,\mathrm{d}p + g\jump{\rho}}{\dsp \frac{\lambda(p_1-p_0)}{3} + \frac{1}{p_1^2} \int_{p_1}^0 p^2 \GammaRel \,\mathrm{d}p - \frac{\sigma}{2}}.
\end{align*}
Rewriting the hypothesis \eqref{size condition air vorticity} gives
\[
	\frac{\sigma}{2} > \frac{\lambda_0^3}{p_1 - p_0} + \frac{\lambda_0 (p_1 - p_0)}{3} + \frac{1}{p_1^2} \int_{p_1}^0 \left( \GammaRel^3 + p^2 \GammaRel \right) \,\mathrm{d}p + g\jump{\rho}.
\]
Then for $|\lambda - \lambda_0|$ small, we have
\[
	\frac{\sigma}{2} > \frac{\lambda^3}{p_1 - p_0} + \frac{\lambda (p_1 - p_0)}{3} + \frac{1}{p_1^2} \int_{p_1}^0 \left( \GammaRel^3 + p^2 \GammaRel \right) \,\mathrm{d}p + g\jump{\rho}.
\]
Recalling that $\varphi$ satisfies inequality \eqref{negative semidefinite}, the above estimate implies that $\mathscr{R}(\varphi; \lambda) > -1$, so \eqref{LBC air vorticity} holds.
\end{proof}
We note that the above proof can be further refined following arguments of \cite[Theorem 4]{Constantin_Strauss2011} to find a smaller lower bound on $\sigma$ guaranteeing \eqref{LBC air vorticity} than that in \eqref{size condition air vorticity}.

\begin{lemma}[monotonicity of $\nu$] \label{Lemma monotonicity of lambda air vorticity}
If $\nu(\lambda) < 0$, then $\nu(\lambda)$ is decreasing in $\lambda$.
\end{lemma}
\begin{proof}
Denoting derivatives with respect to $\lambda$ by a dot, differentiating the eigenvalue problem (\ref{P_mu eigenvalue problem}) with $u = \varphi \in \mathscr{A}$ gives
\begin{equation} \label{P mu dot}
\begin{dcases}
	-(3a^2 \dot{a} \varphi_p)_p - (a^3 \dot{\varphi}_p)_p = \dot{\nu}a\varphi + \nu \dot{a}\varphi + \nu a \dot{\varphi} & \text{in } (p_0, p_1) \cup (p_1, 0), \\
	2\jump{3a^2 \dot{a} \varphi_p + a^3 \dot{\varphi}_p} - 2g\jump{\rho}\dot{\varphi} = \sigma \dot{\nu} \varphi + \sigma \nu \dot{\varphi} & \text{on } p = p_1, \\
	\dot{\varphi} = 0 & \text{on } p = p_0, \\
	\dot{\varphi} = 0 & \text{on } p = 0.
\end{dcases}
\tag{$\dot{P}_\mu$}
\end{equation}
Multiplying (\ref{P_mu eigenvalue problem}) by $\dot{\varphi}$ and integrating yields
\begin{equation} \label{temp eqn 1}
	\int_{p_0}^0 a^3 \varphi_p \dot{\varphi}_p \,\mathrm{d}p + g\jump{\rho}\varphi(p_1)\dot{\varphi}(p_1) + \frac{\sigma}{2}\nu\varphi(p_1)\dot{\varphi}(p_1) = \int_{p_0}^0 \nu a \varphi \dot{\varphi} \,\mathrm{d}p.
\end{equation}
On the other hand, multiplying (\ref{P mu dot}) by $\varphi$ and integrating gives
\begin{align} \label{temp eqn 2}
	\int_{p_0}^0 3a^2\dot{a}\varphi_p^2 \,\mathrm{d}p + g\jump{\rho}\dot{\varphi}(p_1)\varphi(p_1) &+ \frac{\sigma}{2}\dot{\nu}\varphi(p_1)^2 + \frac{\sigma}{2}\nu\dot{\varphi}(p_1)\varphi(p_1) \\
	&+ \int_{p_0}^0 a^3 \dot{\varphi}_p \varphi_p \,\mathrm{d}p = \int_{p_0}^0 \left( \dot{\nu}a\varphi^2 + \nu\dot{a}\varphi^2 + \nu a \dot{\varphi}\varphi \right) \,\mathrm{d}p. \nonumber
\end{align}
Subtracting \eqref{temp eqn 1} from \eqref{temp eqn 2}, we have the following Green's identity
\[
	\int_{p_0}^0 3a^2\dot{a}\varphi_p^2 \,\mathrm{d}p + \frac{\sigma}{2}\dot{\nu}\varphi(p_1)^2 = \int_{p_0}^0 \left( \dot{\nu}a\varphi^2 + \nu\dot{a}\varphi^2 \right) \,\mathrm{d}p.
\]
Since $\dot{a} = \mathbbm{1}_{(p_0, p_1)}$, we can simplify this to find
\[
	\int_{p_0}^{p_1} 3a^2\varphi_p^2 \,\mathrm{d}p - \int_{p_0}^{p_1} \nu \varphi^2 \,\mathrm{d}p = \left( \int_{p_0}^0 a\varphi^2 \,\mathrm{d}p - \frac{\sigma}{2} \varphi(p_1)^2 \right) \dot{\nu}.
\]
Therefore, since $\nu < 0$ by assumption and the quantity in parenthesis is negative, we must have $\dot{\nu} < 0$.
\end{proof}

\begin{lemma} \label{existence maximizer air vorticity}
Suppose that the \eqref{LBC air vorticity} holds. Then there exists a unique value $\lambda^* > 0$ such that $\nu(\lambda^*) = -1$. Equivalently, there exists a unique value of $\lambda$ for which there is a nontrivial solution to the linearized problem (\ref{linearized equation air vorticity}) with the ansatz $m(q,p) = M(p)\cos(q)$. Moreover, $Q$ is an invertible function of $\lambda$ in a neighborhood of $\lambda^*$.
\end{lemma}
\begin{proof}
From Lemma \ref{Lemma Rayleigh quotient < -n^2}, we have $\nu(\lambda) < -1$ for $\lambda$ sufficiently large, and $\nu(\lambda) > -1$ for some $\lambda$ by \eqref{LBC air vorticity}. By continuity, there exists $\lambda^*$ such that $\nu(\lambda^*) = -1$. Moreover, Lemma \ref{Lemma monotonicity of lambda air vorticity} tells us that $\nu$ is a decreasing function when $\nu < 0$, so $\lambda^*$ is unique.

Next, as noted at the end of Lemma \ref{Lemma laminar solution air vorticity}, $Q$ is a concave function of $\lambda$ according to (\ref{Q(lambda) laminar solution air vorticity}), so we only need to show that $\lambda^* \ne \lambda_0$, where $\lambda_0$ is defined in (\ref{lambda_0 formula}) to be the critical point of $Q$. But $\lambda^* > \lambda$ by \eqref{LBC air vorticity}.
\end{proof}

\subsection{Proof of local bifurcation} \label{Section proof of local bifurcation}
We are now prepared to prove Theorem \ref{Theorem local bifurcation air vorticity surface tension}. As stated above, our approach is based on the classical theory of Crandall--Rabinowitz on local bifurcation from simple (generalized) eigenvalues. Specifically, we will treat the family of laminar flows as our trivial solutions. Suppose the solution to the height equation (\ref{height equation air vorticity}) can be decomposed as $h(q,p) = H(p; \lambda) + m(q,p)$ and $Q = Q(\lambda)$. Then substituting it into the equation gives
\[
	\mathcal{F}(\lambda, m) = 0,
\]
where $\mathcal{F} = \left( \mathcal{F}_1, \mathcal{F}_2, \mathcal{F}_3, \mathcal{F}_4 \right) : \Lambda \times \mathcal{O} \to Y$ with $\Lambda \subset \mathbb{R}$ to be a neighborhood of $\lambda^*$, and
\[
	\mathcal{O} := \left\{ m \in X: \inf(m_p + H_p) > 0 \text{ in } \overline D \text{ for all } \lambda \in \Lambda \right\},
\]
\begin{align} \label {mathcal F defn air vorticity}
	\hspace{.4in} \mathcal{F}_1(\lambda, m) &:= (1 + (m^{(1)}_q)^2)(m_{pp}^{(1)} + H_{pp}) + m_{qq}^{(1)}(m_p^{(1)} + H_p)^2 \nonumber \\ & \hspace{.3in} - 2 m_q^{(1)}(H_p + m_p^{(1)}) m_{pq}^{(1)} + \gamma(-p) (H_p + m^{(1)}_p)^3, \nonumber \\
	\mathcal{F}_2(\lambda, m) &:= (1 + (m^{(2)}_q)^2)(m_{pp}^{(2)} + H_{pp}) + m_{qq}^{(2)}(m_p^{(2)} + H_p)^2 \nonumber \\ & \hspace{.3in} - 2 m_q^{(2)}(H_p + m_p^{(2)}) m_{pq}^{(2)}, \\
	\mathcal{F}_3(\lambda, m) &:= -\jump{\frac{1 + m_q^2}{(H_p + m_p)^2}} - 2g\jump{\rho}(m + H) + Q - \frac{\sigma m_{qq}}{(1 + m_q^2)^{3/2}}, \nonumber \\
	\mathcal{F}_4(\lambda, m) &:= \Big( m + H -\ell - d(m) - d(H) \Big)\Big\vert_T. \nonumber
\end{align}
The Banach spaces $X$ and $Y = Y_1 \times Y_2 \times Y_3 \times Y_4$ are defined by
\[
	X := \left\{ h \in C^{2,\alpha}(\overline{D_1}) \cap C^{2,\alpha}(\overline{D_2}) \cap C^{0, \alpha}_{\text{per}}(\overline{D}): h(p_0) = 0 \right\},
\]
\[
	Y_1 := C^{0, \alpha}_{\text{per}}(\overline{D_1}), \quad Y_2 := C^{0,\alpha}_{\text{per}}(\overline{D_2}), \quad 
	Y_3 := C^\alpha_{\text{per}}(I), \quad Y_4 := C^{2,\alpha}_{\text{per}}(T).
\]
It is clear that $\mathcal{F}(\lambda, 0) = 0$ for all $\lambda > 0$. Let us record the Fr\'echet derivative of $\mathcal{F}$ with respect to $m$ at $(\lambda^*, 0)$.
\begin{align*}
	\mathcal{F}_{1m}(\lambda^*,0)\varphi &= \left( \partial_p^2 + H_p^2 \partial_q^2 + 3\gamma H_p^2\partial_p \right) \varphi^{(1)}, \\
	\mathcal{F}_{2m}(\lambda^*,0)\varphi &= \left( \partial_p^2 + H_p^2 \partial_q^2 \right) \varphi^{(2)}, \\
	\mathcal{F}_{3m}(\lambda^*,0)\varphi &= 2\jump{H_p^{-3}\varphi_p} - 2g\jump{\rho}\varphi - \sigma\varphi_{qq}, \\
	\mathcal{F}_{4m}(\lambda^*,0)\varphi &= \Big(\varphi-d(\varphi)\Big)\Big\vert_T.
\end{align*}
Note that in $D_1$, from \eqref{laminar equation air vorticity}, we have $\gamma = -H_{pp} / H_p^3$, so we can write
\begin{align*}
	\mathcal{F}_{1m}(\lambda^*, 0)\varphi &= \varphi_{pp} + a^{-2}\varphi_{qq} + 3H_p \partial_p\left( H_p^{-1} \right)\varphi_p \\
	&= \varphi_{pp} + a^{-2}\varphi_{qq} + a^{-3} 3a^2 (\partial_p a) \varphi_p \\
	&= a^{-3} \left( a^3 \varphi_p \right)_p + a^{-2}\varphi_{qq},
\end{align*}
which is the same quantity as in $D_2$.  Thus, the first expression can be written as
\[
	\mathcal{F}_{im}(\lambda^*, 0)\varphi = a^{-3} \partial_p\left( a^3 \partial_p \varphi^{(i)} \right) + a^{-2} \partial_q^2 \varphi^{(i)} \quad \text{for } i = 1, 2.
\]
\begin{lemma}[null space] \label{Lemma null space air vorticity}
The null space of $\mathcal{F}_m(\lambda^*,0)$ is one-dimensional and spanned by $\varphi^*(q,p) := \varphi_1(p)\cos(q)$.
\end{lemma}
\begin{proof}
Let $\varphi$ be in the null space of $\mathcal{F}_m(\lambda^*,0))$. Since $\varphi$ is even and $C^{0, \alpha}$, we can express it via a cosine series
\[
	\varphi(q,p) = \sum_{n=0}^\infty \varphi_n(p)\cos(nq).
\]
Clearly, $\mathcal{F}_m(\lambda^*,0)\left(\varphi_n(p)\cos(nq)\right) = 0$ for every $n \ge 0$ meaning that $\varphi_n$ must solve (\ref{Pn air vorticity}) for $n$. Since $\lambda^* > \lambda_0$, we can apply Proposition \ref{Proposition unique eigenvalue for operator K} to conclude that the null space of $\mathcal{F}_m(\lambda^*,0)$ is one-dimensional. In particular, it is generated by $\varphi^*(q,p):=\varphi_1(p)\cos(q)$, where $\varphi_1$ is the unique solution to equation (\ref{Pn air vorticity}) for $n=1$.
\end{proof}

Our next lemma characterizes the range of $\mathcal{F}_m(\lambda^*, 0)$.  

\begin{lemma}[range] \label{Lemma range air vorticity}
$\mathcal{A} = (\mathcal{A}_1, \mathcal{A}_2, \mathcal{A}_3, \mathcal{A}_4) \in Y$ is in the range of $\mathcal{F}_m(\lambda^*, 0)$ if and only if it satisfies the following orthogonality condition:
\begin{align} \label{orthogonality condition air vorticity}
	\iint_{D_1} a^3 \mathcal{A}_1 \varphi^* \,\mathrm{d}q\,\mathrm{d}p + \iint_{D_2} a^3 \mathcal{A}_2 \varphi^* \,\mathrm{d}q\,\mathrm{d}p + \frac{1}{2}\int_I \mathcal{A}_3 \varphi^* \,\mathrm{d}q + \int_T a^3 \mathcal{A}_4 \varphi^*_p \mathrm{d}q = 0,
\end{align}
where $\varphi^*$ generates the null space of $\mathcal{F}_m(\lambda^*, 0)$.
\end{lemma}

\begin{proof}
Denoting by $\mathcal{R}(\mathcal{F}_m(\lambda^*, 0))$ the range of $\mathcal{F}_m(\lambda^*, 0)$, we first suppose that $\mathcal{A} \in \mathcal{R}(\mathcal{F}_m(\lambda^*, 0))$. Let $\varphi$ be given such that $\mathcal{F}_m(\lambda^*, 0)\varphi = \mathcal{A}$. Using integration by parts and the PDE for $\varphi$ and $\varphi^*$ ($\varphi^*$ satisfies equation (\ref{Pn air vorticity}) with $n=1$), we can compute
\begin{align*}
	&\left\langle a^3\varphi^*, \mathcal{A}_1\right\rangle_{L^2(D_1)} + \left\langle a^3\varphi^*, \mathcal{A}_2\right\rangle_{L^2(D_2)} \\
	&= \iint_{D_1 \cup D_2} \left( (a^3 \varphi_p)_p + a\varphi_{qq} \right) \varphi^* \,\mathrm{d}q\,\mathrm{d}p \\
	&= -\iint_{D_1 \cup D_2} a^3 \varphi_{p} \varphi^*_p \,\mathrm{d}q\,\mathrm{d}p - \int_I \jump{a^3 \varphi_p \varphi^*} \,\mathrm{d}q + \iint_{D_1 \cup D_2} a \varphi \varphi^*_{qq} \,\mathrm{d}q\,\mathrm{d}p \\
	&= \iint_{D_1 \cup D_2} \left(a^3 \varphi^*_{pp} + a\varphi^*_{qq}\right)\varphi \,\mathrm{d}q\,\mathrm{d}p + \int_I \left(\jump{a^3\varphi^*_p\varphi} - \jump{a^3\varphi_p\varphi^*}\right) \,\mathrm{d}q - \int_T a^3\varphi\varphi^*_p \,\mathrm{d}q.
\end{align*}
Using the jump condition, the fact that $\varphi$ and $\varphi^*$ are continuous across the interface, and integration by parts, we obtain
\begin{align*}
	&\left\langle a^3\varphi^*, \mathcal{A}_1\right\rangle_{L^2(D_1)} + \left\langle a^3\varphi^*, \mathcal{A}_2\right\rangle_{L^2(D_2)} \\
	&= \int_I \left(g\jump{\rho}\varphi^* - \frac{\sigma}{2}\varphi^*\right)\varphi \,\mathrm{d}q - \int_I \left(g\jump{\rho}\varphi + \frac{\sigma}{2}\varphi_{qq} + \frac{1}{2}\mathcal{A}_3\right)\varphi^* \,\mathrm{d}q \\
	&\pushright{- \int_T a^3\mathcal{A}_4\varphi^*_p \,\mathrm{d}q - d(\varphi)\int_T a^3\varphi^*_p \,\mathrm{d}q} \\
	&= \int_I \left(-\frac{\sigma}{2}\varphi^*\varphi - \frac{\sigma}{2}\varphi^*_{qq}\varphi - \frac{1}{2}\mathcal{A}_3\varphi^*\right) \,\mathrm{d}q - \int_T a^3\mathcal{A}_4\varphi^*_p \,\mathrm{d}q.
\end{align*}
Recalling that $\varphi^* = \varphi_1(p) \cos(q)$, the last equality comes from the observation that
\[
	\int_T a^3\varphi^*_p \,\mathrm{d}q = a^3\varphi_{1p}(0) \int_0^{2\pi} \cos(q) \,\mathrm{d}q = 0.
\]
Finally, we have
\[
	\left\langle a^3\varphi^*, \mathcal{A}_1\right\rangle_{L^2(D_1)} + \left\langle a^3\varphi^*, \mathcal{A}_2\right\rangle_{L^2(D_2)} = -\frac{1}{2}\int_I \mathcal{A}_3\varphi^* \,\mathrm{d}q - \int_T a^3\mathcal{A}_4\varphi^*_p \,\mathrm{d}q,
\]
which can be rearranged to yield the identity \eqref{orthogonality condition air vorticity}.

Next, we prove the orthogonality condition \eqref{orthogonality condition air vorticity} is sufficient.  By similar arguments as in \cite[Lemma 3.6]{Wahlen_gravity2006}, we define the following inner product for each $\lambda$:
\begin{align*}
	\langle (\mathcal{U}_1, \mathcal{U}_2, \mathcal{U}_3, \mathcal{U}_4), & (\mathcal{V}_1, \mathcal{V}_2, \mathcal{V}_3, \mathcal{V}_4) \rangle_Y := \\
	&\left\langle a^3 \mathcal{U}_1, \mathcal{V}_1 \right\rangle_{L^2(D_1)} + \left\langle a^3 \mathcal{U}_2, \mathcal{V}_2 \right\rangle_{L^2(D_2)} + \frac{1}{2} \left\langle \mathcal{U}_3, \mathcal{V}_3 \right\rangle_{L^2(I)} + \left\langle a^3 \mathcal{U}_4, \mathcal{V}_4 \right\rangle_{L^2(T)}.
\end{align*}
Note that the null space $\widehat{\mathcal{N}}$ of $\mathcal{F}_m(\lambda^*, 0)$ can be identified with the subspace
\[
	\widetilde{\mathcal{N}} := \{ (\mathcal{V}_1, \mathcal{V}_2, \mathcal{V}_3, \mathcal{V}_4) \in Y: \mathcal{V}_1 = \mathcal{V}\vert_{D_1}, \; \mathcal{V}_2 = \mathcal{V}\vert_{D_2}, \; \mathcal{V}_3 = \mathcal{V}\vert_I, \; \mathcal{V}_4 = \mathcal{V}\vert_T \text{ for some } \mathcal{V} \in \widehat{\mathcal{N}} \}.  
\]
Then the necessary condition, which is shown above, implies that if $\mathcal{A}$ is in the range of $\mathcal{F}_m(\lambda^*, 0)$, then
\[
	\left\langle \left( \varphi^*|_{D_1}, \varphi^*|_{D_2}, \varphi^*|_{I}, \varphi^*|_{T} \right), \mathcal{A} \right\rangle_Y = 0,
\]
so $\mathcal{R}(\mathcal{F}_m(\lambda^*, 0)) \subset \widetilde{\mathcal{N}}^\perp$.  By Remark \ref{Remark Fredholm index 0}, $\mathcal{F}_m(\lambda^*, 0)$ has Fredholm index $0$, and hence
\[
	\mathrm{codim}\, \mathcal{R}(\mathcal{F}_m(\lambda^*, 0)) = \dim \widehat{\mathcal{N}} = \dim \widetilde{\mathcal{N}} = \mathrm{codim}\, \widetilde{\mathcal{N}}^\perp < \infty,
\]
which means that $\mathcal{R}(\mathcal{F}_m(\lambda^*, 0)) = \widetilde{\mathcal{N}}^\perp$.  This concludes the proof of the lemma.  
\end{proof}

\begin{lemma} [transversality] \label{Lemma transversality air vorticity}
The following transversality condition holds
\begin{equation} \label{transversality condition air vorticity}
	\mathcal{F}_{\lambda m} (\lambda^*, 0) \varphi^* \notin \mathcal{R}(\mathcal{F}_m(\lambda^*, 0)),
\end{equation}
where $\varphi^*$ generates the null space of $\mathcal{F}_m(\lambda^*, 0)$ solves equation \eqref{Pn air vorticity} for $n=1$.
\end{lemma}
\begin{proof}
Using Lemma \ref{Lemma range air vorticity}, it suffices to show that $\mathcal{A} := \mathcal{F}_{\lambda m}(\lambda^*, 0) \varphi^*$ does not satisfy the orthogonality condition \eqref{orthogonality condition air vorticity}. We must confirm that
\[
	\Xi := \iint_{D_1} a^3\mathcal{A}_1\varphi^* \,\mathrm{d}q \,\mathrm{d}p + \iint_{D_2} a^3\mathcal{A}_2\varphi^* \,\mathrm{d}q \,\mathrm{d}p + \frac{1}{2}\int_I \mathcal{A}_3\varphi^* \,\mathrm{d}q + \int_T a^3\mathcal{A}_4\varphi^*_p \,\mathrm{d}q \ne 0.
\]
Computing the derivatives gives
\begin{align*}
	\mathcal{F}_{1\lambda m}(\lambda^*, 0)\varphi^* &= 0, \\
	\mathcal{F}_{2\lambda m}(\lambda^*, 0)\varphi^* &= -\frac{2}{(\lambda^*)^3} \varphi^*_{qq}, \\
	\mathcal{F}_{3\lambda m}(\lambda^*, 0)\varphi^* &= \left.\left( -6(\lambda^*)^2 \left(\varphi^*_p\right)^{(2)} \right) \right\vert_I, \\
	\mathcal{F}_{4\lambda m}(\lambda^*, 0)\varphi^* &= 0,
\end{align*}
so we have
\[
	\iint_{D_1} a^3\mathcal{A}_1\varphi^* \,\mathrm{d}q \,\mathrm{d}p = \int_T a^3\mathcal{A}_4\varphi^*_p \,\mathrm{d}q = 0.
\]
Using \eqref{Pn air vorticity} with $n=1$ to have $\varphi^* = (\lambda^*)^2 \varphi^*_{pp}$ , we can derive
\[
	\left(\lambda^*\right)^2 \left(\varphi^*\varphi^*_p\right)_p = \left(\lambda^*\right)^2\left(\varphi^*_p\right)^2 + \left(\lambda^*\right)^2\varphi^*\varphi^*_{pp} = \left(\lambda^*\right)^2\left(\varphi^*_p\right)^2 + \left(\varphi^*\right)^2 \quad \text{in } D_2
\]
so that using integration by parts and the fact that $\varphi^*_{qq} = -\varphi^*$ yields
\begin{align*}
	\iint_{D_2} a^3\mathcal{A}_2\varphi^* \,\mathrm{d}q \,\mathrm{d}p &= \iint_{D_2} 2 (\varphi^*)^2 \,\mathrm{d}q \,\mathrm{d}p
\end{align*}
and
\begin{align*}
	\frac{1}{2}\int_I \mathcal{A}_3\varphi^* \,\mathrm{d}q &= -3\left(\lambda^*\right)^2 \int_I \left(\varphi^*_p\right)^{(2)}\varphi^* \,\mathrm{d}q \\
	&= -3\left(\lambda^*\right)^2 \iint_{D_2} \left(\varphi^*_p\right)^2 \,\mathrm{d}q \,\mathrm{d}p - 3 \iint_{D_2} \left(\varphi^*\right)^2 \,\mathrm{d}q \,\mathrm{d}p.
\end{align*}
Finally, combining all terms gives
\begin{align*}
	\Xi &= -\iint_{D_2} \left(\varphi^*\right)^2 \,\mathrm{d}q \,\mathrm{d}p - 3\left(\lambda^*\right)^2 \iint_{D_2} \left(\varphi^*_p\right)^2 \,\mathrm{d}q \,\mathrm{d}p < 0. \qedhere
\end{align*}
\end{proof}
Now we are ready to prove our main theorem.

\begin{proof}[Proof of Theorem \ref{Theorem local bifurcation air vorticity surface tension}]
Suppose conditions \eqref{Gamma_rel definition air vorticity} and \eqref{LBC air vorticity} are satisfied. Then $\mathcal{F}(\lambda, 0) = 0$ for all $\lambda > 0$ and $\mathcal{F}_m$, $\mathcal{F}_\lambda$, $\mathcal{F}_{\lambda m}$ exist and are continuous, which means parts $(i)$ and $(ii)$ are confirmed. Moreover, Lemma \ref{Lemma null space air vorticity} and Lemma \ref{Lemma range air vorticity} give dimension $1$ for the null space of $\mathcal{F}_m(\lambda^*, 0)$ and co-dimension $1$ for the range of $\mathcal{F}_m(\lambda^*, 0)$. Thus, $\mathcal{F}_m(\lambda^*, 0)$ has Fredholm index $0$, and hence part $(iii)$ is justified. Lastly, the transversality condition in Lemma \ref{Lemma transversality air vorticity} fulfills part $(iv)$. Therefore, the local bifurcation result follows directly from Theorem \ref{Theorem Crandall Bifurcation}.
\end{proof}
Finally, back to our objective, we note that the existence of a solution of class $\solnSpace$ to the height equation \eqref{height equation air vorticity} is equivalent to the existence of a solution of class $\solnSpaceEuler$ to the Euler system \eqref{Eulerian eqn 1}--\eqref{wo stagnation} (see, for example, \cite[Lemma 2.1]{Constantin_Strauss2004} or \cite[Lemma 2.1]{Walsh2009}).

\section*{Acknowledgements}
The author would like to express his sincere gratitude to \thanks{Samuel Walsh} for the continuous support and insightful comments.  The author is also grateful to \thanks{Erik Wahl\'en} for assisting with the Pontryagin space material.  Finally, the author is deeply indebted to the referee for many helpful comments and suggestions which significantly improved the paper.  

This work was supported in part by the National Science Foundation through DMS-$1514910$.


\renewcommand{\theequation}{A-\arabic{equation}}
\renewcommand{\thetheorem}{A.\arabic{theorem}}

\appendix
\setcounter{theorem}{0}
\setcounter{equation}{0}
\subsection*{Appendix}
\begin{theorem} [Crandall and Rabinowitz \cite{Crandall_Rabinowitz1971}] \label{Theorem Crandall Bifurcation}
Let $X$ and $Y$ be Banach spaces and $I \subset \mathbb{R}$ be an open interval with $\lambda^* \in I$. Suppose that $\mathcal{F}: I \times X \to Y$ is a continuous map with the following properties:
\begin{enumerate}[label=(\roman*),font=\upshape]
	\item $\mathcal{F}(\lambda,0) = 0$ for all $\lambda \in I$;
	\item $D_1\mathcal{F}$, $D_2\mathcal{F}$, and $D_1D_2\mathcal{F}$ exist and are continuous, where $D_i$ denotes the Fr\'echet derivative with respect to the i-th coordinate;
	\item $D_2\mathcal{F}(\lambda^*,0)$ is a Fredholm operator of index $0$. In particular, the null space is one-dimensional and spanned by some element $w^*$.
	\item $D_1D_2\mathcal{F}(\lambda^*,0)w^* \notin \mathcal{R}(D_2\mathcal{F}(\lambda^*,0))$.
\end{enumerate}
There there exists a continuous local bifurcation curve $\{ (\lambda(s),w(s)) \in \mathbb{R} \times X: |s| < \epsilon \}$ with $\epsilon>0$ sufficiently small such that $(\lambda(0),w(0)) = (\lambda^*,w^*)$, and
\[
	\{ (\lambda,w) \in \mathcal{U}: w \ne 0, \mathcal{F}(\lambda,w) = 0 \} = \{ (\lambda(s),w(s)) \in \mathbb{R} \times Y: |s| < \epsilon \}
\]
for some neighborhood $\mathcal{U}$ of $(\lambda^*,0)$ in $\mathbb{R} \times X$. Moreover, we have
\[
	w(s) = sw^* + o(s) \qquad \text{in } X \text{, } |s| < \epsilon.
\]
If $D^2_2$ exists and is continuous, then the curve is of class $C^1$.
\end{theorem}

\bibliography{Elliptic_Wentzell_Transimission}

\begin{thebibliography}{10}

\bibitem{Amick_Turner1986}
C.~J. Amick and R.~E.~L. Turner.
\newblock A global theory of internal solitary waves in two-fluid systems.
\newblock {\em Trans. Amer. Math. Soc.}, 298(2):431--484, 1986.

\bibitem{Amick1984}
Charles~J. Amick.
\newblock Semilinear elliptic eigenvalue problems on an infinite strip with an
  application to stratified fluids.
\newblock {\em Ann. Scuola Norm. Sup. Pisa Cl. Sci. (4)}, 11(3):441--499, 1984.

\bibitem{Apushkinskaya_Nazarov2000}
D.~E. Apushkinskaya and A.~I. Nazarov.
\newblock A survey of results on nonlinear {V}enttsel problems.
\newblock {\em Appl. Math.}, 45(1):69--80, 2000.

\bibitem{Apushkinskaya_Nazarov2001}
Darya~E. Apushkinskaya and Aleksandr~I. Nazarov.
\newblock Linear two-phase {V}enttsel problems.
\newblock {\em Ark. Mat.}, 39(2):201--222, 2001.

\bibitem{Arkhipova_Erlhamahmy2002}
A.~A. Arkhipova and Osman Erlhamahmy.
\newblock Regularity of solutions to a diffraction-type problem for nondiagonal
  linear elliptic systems in the {C}ampanato space.
\newblock {\em J. Math. Sci. (New York)}, 112(1):3944--3966, 2002.
\newblock Function theory and applications.

\bibitem{Beale_Hou_Lowengrub1993}
J.~Thomas Beale, Thomas~Y. Hou, and John~S. Lowengrub.
\newblock Growth rates for the linearized motion of fluid interfaces away from
  equilibrium.
\newblock {\em Comm. Pure Appl. Math.}, 46(9):1269--1301, 1993.

\bibitem{Bognar1974}
J\'anos Bogn\'ar.
\newblock {\em Indefinite Inner Product Spaces}.
\newblock Springer-Verlag, New York-Heidelberg, 1974.
\newblock Ergebnisse der Mathematik und ihrer Grenzgebiete, Band 78.

\bibitem{Bona_Bose_Turner1983}
J.~L. Bona, D.~K. Bose, and R.~E.~L. Turner.
\newblock Finite-amplitude steady waves in stratified fluids.
\newblock {\em J. Math. Pures Appl. (9)}, 62(4):389--439 (1984), 1983.

\bibitem{Bonnaillie_Dambrine2010}
V.~Bonnaillie-No\"el, M.~Dambrine, F.~H\'erau, and G.~Vial.
\newblock On generalized {V}entcel's type boundary conditions for {L}aplace
  operator in a bounded domain.
\newblock {\em SIAM J. Math. Anal.}, 42(2):931--945, 2010.

\bibitem{Borsuk2008}
Mikhail Borsuk.
\newblock The transmission problem for elliptic second order equations in a
  conical domain.
\newblock {\em Ann. Acad. Pedagog. Crac. Stud. Math.}, 7:61--89, 2008.

\bibitem{Borsuk2009}
Mikhail Borsuk.
\newblock The transmission problem for quasi-linear elliptic second order
  equations in a conical domain. {I}, {II}.
\newblock {\em Nonlinear Anal.}, 71(10):5032--5083, 2009.

\bibitem{Borsuk2010}
Mikhail Borsuk.
\newblock {\em Transmission problems for elliptic second-order equations in
  non-smooth domains}.
\newblock Frontiers in Mathematics. Birkh\"auser/Springer Basel AG, Basel,
  2010.

\bibitem{Buhler_Shatah_Walsh2013}
Oliver B\"uhler, Jalal Shatah, and Samuel Walsh.
\newblock Steady water waves in the presence of wind.
\newblock {\em SIAM J. Math. Anal.}, 45(4):2182--2227, 2013.

\bibitem{Buhler_Shatah_Walsh_Zeng2016}
Oliver B\"uhler, Jalal Shatah, Samuel Walsh, and Chongchun Zeng.
\newblock On the wind generation of water waves.
\newblock {\em Arch. Ration. Mech. Anal.}, 222(2):827--878, 2016.

\bibitem{Chen_Walsh2016}
Robin~Ming Chen and Samuel Walsh.
\newblock Continuous dependence on the density for stratified steady water
  waves.
\newblock {\em Arch. Ration. Mech. Anal.}, 219(2):741--792, 2016.

\bibitem{Constantin2011}
Adrian Constantin.
\newblock {\em Nonlinear water waves with applications to wave-current
  interactions and tsunamis}, volume~81 of {\em CBMS-NSF Regional Conference
  Series in Applied Mathematics}.
\newblock Society for Industrial and Applied Mathematics (SIAM), Philadelphia,
  PA, 2011.

\bibitem{Constantin_Strauss2004}
Adrian Constantin and Walter Strauss.
\newblock Exact steady periodic water waves with vorticity.
\newblock {\em Comm. Pure Appl. Math.}, 57(4):481--527, 2004.

\bibitem{Constantin_Strauss2011}
Adrian Constantin and Walter Strauss.
\newblock Periodic traveling gravity water waves with discontinuous vorticity.
\newblock {\em Arch. Ration. Mech. Anal.}, 202(1):133--175, 2011.

\bibitem{Crandall_Rabinowitz1971}
Michael~G. Crandall and Paul~H. Rabinowitz.
\newblock Bifurcation from simple eigenvalues.
\newblock {\em J. Functional Analysis}, 8:321--340, 1971.

\bibitem{Dubreil-Jacotin1934}
M.-L. Dubreil-Jacotin.
\newblock {\em Sur la d\'etermination rigoureuse des ondes permanentes
  p\'eriodiques d'ampleur finie}.
\newblock 1934.

\bibitem{Gilbarg_Trudinger2001}
David Gilbarg and Neil~S. Trudinger.
\newblock {\em Elliptic Partial Differential Equations of Second Order}.
\newblock Classics in Mathematics. Springer-Verlag, Berlin, 2001.
\newblock Reprint of the 1998 edition.

\bibitem{Ikeda_Watanabe1989}
Nobuyuki Ikeda and Shinzo Watanabe.
\newblock {\em Stochastic Differential Equations and Diffusion Processes},
  volume~24 of {\em North-Holland Mathematical Library}.
\newblock North-Holland Publishing Co., Amsterdam; Kodansha, Ltd., Tokyo,
  second edition, 1989.

\bibitem{Iohvidov1982}
I.~S. Iohvidov, M.~G. Krein, and H.~Langer.
\newblock {\em Introduction to the Spectral Theory of Operators in Spaces with
  an Indefinite Metric}, volume~9 of {\em Mathematical Research}.
\newblock Akademie-Verlag, Berlin, 1982.

\bibitem{Kirchgassner1982}
Klaus Kirchg\"assner.
\newblock Wave-solutions of reversible systems and applications.
\newblock {\em J. Differential Equations}, 45(1):113--127, 1982.

\bibitem{Korman1986}
Philip Korman.
\newblock Existence of solutions for a class of semilinear noncoercive
  problems.
\newblock {\em Nonlinear Anal.}, 10(12):1471--1476, 1986.

\bibitem{Ladyzhenskaya1968}
Olga~A. Ladyzhenskaya and Nina~N. Ural'tseva.
\newblock {\em Linear and Quasilinear Elliptic Equations}.
\newblock Academic Press, New York-London, 1968.

\bibitem{Lankers_Friesecke1997}
Katharina Lankers and Gero Friesecke.
\newblock Fast, large-amplitude solitary waves in the {$2$}{D} {E}uler
  equations for stratified fluids.
\newblock {\em Nonlinear Anal.}, 29(9):1061--1078, 1997.

\bibitem{Luo1991}
Yousong Luo.
\newblock On the quasilinear elliptic {V}enttsel$\prime$\ boundary value
  problem.
\newblock {\em Nonlinear Anal.}, 16(9):761--769, 1991.

\bibitem{Luo_Trudinger1991}
Yousong Luo and Neil~S. Trudinger.
\newblock Linear second order elliptic equations with {V}enttsel boundary
  conditions.
\newblock {\em Proc. Roy. Soc. Edinburgh Sect. A}, 118(3-4):193--207, 1991.

\bibitem{Luo_Trudinger1994}
Yousong Luo and Neil~S. Trudinger.
\newblock Quasilinear second order elliptic equations with {V}enttsel boundary
  conditions.
\newblock {\em Potential Anal.}, 3(2):219--243, 1994.

\bibitem{Martin_Matioc2014}
Calin~Iulian Martin and Bogdan-Vasile Matioc.
\newblock Existence of capillary-gravity water waves with piecewise constant
  vorticity.
\newblock {\em J. Differential Equations}, 256(8):3086--3114, 2014.

\bibitem{Matioc_Matioc2014}
Anca-Voichita Matioc and Bogdan-Vasile Matioc.
\newblock Capillary-gravity water waves with discontinuous vorticity: existence
  and regularity results.
\newblock {\em Comm. Math. Phys.}, 330(2):859--886, 2014.

\bibitem{Miles1957}
John~W. Miles.
\newblock On the generation of surface waves by shear flows.
\newblock {\em J. Fluid Mech.}, 3:185--204, 1957.

\bibitem{Nazarov_Paletskikh2015}
A.~I. Nazarov and A.~A. Paletskikh.
\newblock On the {H}\"older property of the solutions of the elliptic
  {V}enttsel problem.
\newblock {\em Dokl. Akad. Nauk}, 465(no.~5):532--536, 2015.

\bibitem{Nilsson2017}
Dag~Viktor Nilsson.
\newblock Internal gravity-capillary solitary waves in finite depth.
\newblock {\em Math. Methods Appl. Sci.}, 40(4):1053--1080, 2017.

\bibitem{Oleinik1961}
O.~A. Ole\u{\i}nik.
\newblock Boundary-value problems for linear equations of elliptic parabolic
  type with discontinuous coefficients.
\newblock {\em Izv. Akad. Nauk SSSR Ser. Mat.}, 25:3--20, 1961.

\bibitem{Schechter1960}
Martin Schechter.
\newblock A generalization of the problem of transmission.
\newblock {\em Ann. Scuola Norm. Sup. Pisa (3)}, 14:207--236, 1960.

\bibitem{Strauss2010}
Walter~A. Strauss.
\newblock Steady water waves.
\newblock {\em Bull. Amer. Math. Soc. (N.S.)}, 47(4):671--694, 2010.

\bibitem{Sun1997}
S.~M. Sun.
\newblock Existence of solitary internal waves in a two-layer fluid of infinite
  depth.
\newblock In {\em Proceedings of the {S}econd {W}orld {C}ongress of {N}onlinear
  {A}nalysts, {P}art 8 ({A}thens, 1996)}, volume~30, pages 5481--5490, 1997.

\bibitem{Sun2002}
S.~M. Sun.
\newblock Solitary internal waves in continuously stratified fluids of great
  depth.
\newblock {\em Phys. D}, 166(1-2):76--103, 2002.

\bibitem{Turner1981}
R.~E.~L. Turner.
\newblock Internal waves in fluids with rapidly varying density.
\newblock {\em Ann. Scuola Norm. Sup. Pisa Cl. Sci. (4)}, 8(4):513--573, 1981.

\bibitem{Seftel1963}
Z.~G. \u{S}eftel'.
\newblock Estimates in {$L_{p}$} of solutions of elliptic equations with
  discontinuous coefficients and satisfying general boundary conditions and
  conjugacy conditions.
\newblock {\em Soviet Math. Dokl.}, 4:321--324, 1963.

\bibitem{Ventcel1959}
A.~D. Ventcel${}^\prime$.
\newblock On boundary conditions for multi-dimensional diffusion processes.
\newblock {\em Theor. Probability Appl.}, 4:164--177, 1959.

\bibitem{Wahlen_gravity2006}
Erik Wahl\'en.
\newblock Steady periodic capillary-gravity waves with vorticity.
\newblock {\em SIAM J. Math. Anal.}, 38(3):921--943, 2006.

\bibitem{Wahlen2006}
Erik Wahl\'en.
\newblock Steady periodic capillary waves with vorticity.
\newblock {\em Ark. Mat.}, 44(2):367--387, 2006.

\bibitem{Walsh2009}
Samuel Walsh.
\newblock Stratified steady periodic water waves.
\newblock {\em SIAM J. Math. Anal.}, 41(3):1054--1105, 2009.

\bibitem{Walsh2014}
Samuel Walsh.
\newblock Steady stratified periodic gravity waves with surface tension {I}:
  {L}ocal bifurcation.
\newblock {\em Discrete Contin. Dyn. Syst.}, 34(8):3241--3285, 2014.

\bibitem{Walsh2014_2}
Samuel Walsh.
\newblock Steady stratified periodic gravity waves with surface tension {II}:
  global bifurcation.
\newblock {\em Discrete Contin. Dyn. Syst.}, 34(8):3287--3315, 2014.

\bibitem{Wilton1915}
J.~R. Wilton.
\newblock On ripples.
\newblock {\em Phil.Mag.}, 29, 1915.

\end{thebibliography}
\bibliographystyle{plain}

\end{document}